\let\accentvec\vec

\documentclass[smallextended]{svjour3period}

\let\vec\accentvec
\usepackage{amsmath}
\usepackage{amsxtra}
\usepackage{amssymb}
\usepackage{lineno}
\usepackage{enumitem}
\usepackage{centernot}
\usepackage[colorlinks=true,linkcolor=blue,citecolor=blue,urlcolor=red]{hyperref}

\newcommand{\mouseM}{\mathcal{M}}
\newcommand{\cpmR}{S}
\newcommand{\construction}{\mathbb{D}}
\newcommand{\conlength}{\Lambda}
\newcommand{\conlengthabstract}{\Gamma}
\newcommand{\com}{\circ}
\newcommand{\Her}{\mathcal{H}}
\newcommand{\es}{\mathbb{E}}

\newcommand{\CC}{\mathbb C}
\newcommand{\RR}{\mathbb R}
\newcommand{\sub}{\subseteq}
\newcommand{\inter}{\cap}

\newcommand{\om}{\omega}
\newcommand{\pow}{\mathcal{P}}
\newcommand{\Tt}{\mathcal{T}}
\newcommand{\Uu}{\mathcal{U}}
\newcommand{\cut}{\backslash}
\newcommand{\ins}{\trianglelefteq}
\newcommand{\pins}{\triangleleft}
\newcommand{\crit}{\mathrm{crit}}
\newcommand{\rest}{\!\upharpoonright\!}
\newcommand{\lh}{\mathrm{lh}}
\newcommand{\Ult}{\mathrm{Ult}}
\newcommand{\sats}{\models}
\newcommand{\J}{\mathcal{J}}
\newcommand{\ZFC}{\mathsf{ZFC}}
\newcommand{\eps}{\varepsilon}
\newcommand{\Ttvec}{{\vec{\Tt}}}

\newcommand{\muvec}{\vec{\mu}}
\newcommand{\core}{\mathfrak{C}}
\newcommand{\her}{\mathcal{H}}
\newcommand{\pred}{\mathrm{-pred}}
\newcommand{\un}{\cup}
\newcommand{\id}{\mathrm{id}}

\newcommand{\conc}{\ \widehat{\ }\ }
\newcommand{\vect}{\vec}
\newcommand{\lex}{\text{lex}}
\newcommand{\unsq}{\text{unsq}}
\newcommand{\rg}{\text{range}}
\newcommand{\Pp}{\mathcal{P}}
\newcommand{\Ee}{\mathcal{E}}
\newcommand{\rhovec}{\vect{\rho}}
\newcommand{\sigmavec}{\vect{\sigma}}

\newcommand{\OR}{\mathsf{ord}}
\renewcommand{\Ttvec}{\vect{\Tt}}

\renewcommand{\muvec}{\vect{\mu}}
\newcommand{\I}{\mathcal{I}}

\spnewtheorem{thm}{Theorem}[section]{\bf}{\rm}
\spnewtheorem{dfn}[thm]{Definition}{\bf}{\rm}
\spnewtheorem{prop}[thm]{Proposition}{\bf}{\rm}
\spnewtheorem{cor}[thm]{Corollary}{\bf}{\rm}
\spnewtheorem{clm}{Claim}{\it}{\rm}
\spnewtheorem*{clm*}{Claim}{\it}{\rm}
\spnewtheorem{rem}[thm]{Remark}{\bf}{\rm}
\spnewtheorem{sclm}{Subclaim}{\it}{\rm}
\spnewtheorem*{sclm*}{Subclaim}{\it}{\rm}

\title{Comparison of fine structural mice via coarse iteration\footnote{Originally published 30 
April, 2014, in Archive for Mathematical Logic, Volume 53, Issue 5-6, pp. 539-559.
The final publication is available at 
\url{link.springer.com}. See \url{http://link.springer.com/article/10.1007\%2Fs00153-014-0379-6}}}
\author{F. Schlutzenberg\and J. R. Steel}
\institute{F. Schlutzenberg
          \at Miami University, Oxford,
Ohio\\\email{farmer.schlutzenberg@gmail.com}
          \and
          J. R. Steel
          \at University of California, Berkeley, California\\\email{steel@math.berkeley.edu}}

\smartqed
\begin{document}
\maketitle
\abstract{Let $\mouseM$ be a fine structural mouse. Let $\construction$ be a
fully backgrounded $L[\es]$-con\-struc\-t\-ion
computed inside an iterable coarse premouse $S$. We describe a process
comparing $\mouseM$ with $\construction$,
through forming iteration trees on $\mouseM$ and on $S$. We then prove that this
process succeeds.
}

\section{Introduction}\label{sec:intro}
Let $\mouseM$ be a fine structural mouse. Let
$\construction=\left<N_\alpha\right>_{\alpha\leq\conlength}$ be a fully
backgrounded $L[\es]$-construction
\footnote{That is, a background extender construction using total background
extenders, similar to that defined in \cite[\S 11]{fsit}.} computed inside an
iterable coarse premouse $S$. In certain situations, one would like
to
compare $\mouseM$ with $\construction$, carrying along the universe $S$.
For example, one might want
to form an iteration tree $\Tt$ on $\mouseM$, with last
model $\mouseM'$, and an iteration tree $\Uu$ on $S$, also with a last
model, such that either
$i^\Uu(N_\conlength)\ins\mouseM'$ or $\mouseM'=N_\alpha^{i^\Uu(\construction)}$
for some $\alpha$. (Here $\Tt$ is fine, as in
\cite[Chap. 5]{fsit}, and $\Uu$ is coarse, as in \cite{it}.) We give
details of such a comparison here, making fairly minimal assumptions about the
$L[\es]$-construction. This sort of comparison is used (without explanation of the details) in \cite{derived}.
\footnote{See the proofs of Corollary 14.3 and Theorem 16.1 of \cite{derived}.}
\footnote{A related problem is that of comparing (the outputs
of) two $L[\es]$-con\-struc\-t\-ions
$\CC^R,\CC^S$, computed inside coarsely iterable universes $R,S$, through
forming coarse iteration trees on $R,S$. This problem presents somewhat
different challenges, and will be dealt with in a separate paper.}

\emph{Notation \& Definitions:} Given a transitive structure $M$,
we use both $\OR^M$ and $\OR(M)$ for $\OR\inter M$. Likewise for other classes of $M$.
See \cite{it} for the definition of \emph{coarse
premouse}, and \cite{outline} for \emph{premouse}. Let $M$ be a premouse and
let $\alpha\leq\OR^M$ be a limit. We write $M|\alpha$ for the $P$ such that
$\OR^P=\alpha$ and
$P\ins M$, and we write $M||\alpha$ for its passive counterpart. We write $F^M$
for the active extender of $M$, $\es^M$ for the extender sequence of $M$,
\emph{not} including $F^M$, and $\es_+^M$ for $\es^M\conc F^M$. We
write $\lh$ for length, and $\nu(F)$ denotes the natural length
of an extender $F$. We say $M$ is \emph{typical} iff
condensation holds for the proper segments (i.e., proper
initial segments) of $M$, and
\cite[4.11, 4.12, 4.15]{thesis} hold for $M$ and its proper segments. (These
properties are consequences of $(0,\om_1,\om_1+1)$-iterability.) Given a
squashed premouse $N$, we
write $N^\unsq$ for the unsquash of $N$; if $N$ is a premouse, we let
$N^\unsq=N$.
Given an iteration tree $\Tt$ of successor length $\theta+1$, we
write $b^\Tt$ for $[0,\theta]_\Tt$ and $\I^\Tt$ for $M^\Tt_\theta$. Given
an extender $E^*$, we write $\rho(E^*)$ for the strength of $E^*$, i.e. the
largest $\rho$ such that $V_\rho\sub\Ult(V,E^*)$.  Let $R$ be a coarse premouse
and $\Uu$ a putative iteration tree on $R$. We say $\Uu$ is \emph{strictly
strength increasing} iff for every $\alpha+1<\beta+1<\lh(\Uu)$ we have
$\rho^{M^\Uu_\alpha}(E^\Uu_\alpha)<\rho^{M^\Uu_\beta}(E^\Uu_\beta)$; $\Uu$ is
\emph{nonoverlapping} iff for every $\alpha+1<\lh(\Uu)$, $\Uu\pred(\alpha+1)$
is the least $\gamma\leq\alpha$ such that
$\crit(E^\Uu_\alpha)<\rho(E^\Uu_\delta)$ for all $\delta\in[\gamma,\alpha)$;
$\Uu$ is \emph{normal} iff it is strictly strength increasing and
nonoverlapping. Given an iteration tree $\Tt$, we write
$\Tt\conc\Pp$ for an extension of $\Tt$ consisting only of padding
$\Pp=\left<\emptyset,\emptyset,\ldots\right>$; here $\lh(\Pp)$
is determined by context. We consider $\emptyset$ as the trivial extender, with
$\Ult(V,\emptyset)=V$ and $i_\emptyset=\id$. We write
$\nu(\emptyset)=\rho(\emptyset)=\infty$.

\section{Main result}
\begin{dfn}[Construction]\label{dfn:construction} Suppose $V=(|V|,\delta)$ is a coarse
premouse.\footnote{We don't assume $V\sats\ZFC$ here; we're just working inside
some coarse premouse.}
Let $x\in\RR$. We say $\CC$ is an \emph{$L[\es,x]$-construction} iff:
\begin{itemize}
 \item[(a)] $\CC=\left<N_\alpha,E^*_\alpha\right>_{\alpha\leq\lambda}$ is a
sequence of $x$-premice $N_\alpha$ and extenders $E^*_\alpha\in V_\delta$
(possibly $E^*_\alpha=\emptyset$);
 \item[(b)] $N_0=\J_1(x)$;
 \item[(c)] For limit $\eta\leq\lambda$, $N_\eta$ is the lim inf of
$\left<N_\gamma\right>_{\gamma<\eta}$;
 \item[(d)] Let $\alpha<\lambda$. Either
\begin{itemize}
\item[(i)] $N_{\alpha+1}=\J_1(\core_\om(N_\alpha))$; or
\item[(ii)] $N_\alpha$ is passive, $N_{\alpha+1}$ is active with
$N_{\alpha+1}=(N_\alpha,F)$ for some $F$, and $F\rest\nu(F)\sub
E^*_{\alpha+1}$.
\end{itemize}
\end{itemize}
Let $\CC=\left<N_\alpha,E^*_\alpha\right>_{\alpha\leq\lambda}$ be an
$L[\es,x]$-construction.
Given $\alpha+1\leq\lambda$ such that $N=N_{\alpha+1}$ is active, $F=F^N$ and
$E^*=E^*_{\alpha+1}$,
we associate an extender $F^*_{\alpha+1}=E^*\rest\beta$ where $\beta$ is least
such that $\beta\geq\nu(F)$ and $\rho(E^*\rest\beta)\geq\min(\rho(E^*),\nu(F))$.
Then in fact, $\rho(E^*\rest\beta)=\min(\rho(E^*),\nu(F))$.
(We have $\nu(F)\leq\lh(E^*)$ since $F\rest\nu(F)\sub E^*$.)

We say $\CC$ is (e) \emph{strongly reasonable}, (f) \emph{reasonable}, (g)
\emph{normal} iff for all $\alpha<\lambda$, if $N=N_{\alpha+1}$ is active then
\begin{itemize}
\item[(e)] \emph{Strong reasonableness}: For all $\kappa<\nu(F^N)$, if
$N\sats$``$\kappa$ is inaccessible'' then $\kappa<\rho(E^*_{\alpha+1})$.
\item[(f)] \emph{Reasonableness}: Let $\lambda$ be the largest limit cardinal of
$N$, let $\eta=(\lambda^+)^N$ and $\kappa<\nu(F^N)$. Suppose $\kappa$ is
measurable in $U=\Ult(V,E^*_{\alpha+1})$ and for every $\xi<\eta$ there are
$E',N'\in U$ such that $U\sats$``$E'$ is an extender with $\crit(E')=\kappa$'',
$N'$ is an active premouse, $\OR^{N'}>\xi$, either $N'\pins N$ or
$N'||\eta=N||\eta$, and $F^{N'}\rest\nu(F^{N'})\sub E'$. Then
$\kappa<\rho(E^*_{\alpha+1})$.
\item[(g)] \emph{Normality}: (i) $E^*_{\alpha+1}\in\her_{|V_\rho|+1}$, where
$\rho=\rho(E^*_{\alpha+1})$; (ii)
For all $\kappa$ such that $\nu(F^{N_{\alpha+1}})<\kappa<\rho(E^*_{\alpha+1})$,
we have $\Ult(V,E^*_{\alpha+1})\sats$``$\kappa$ is not measurable''.
\end{itemize}
\end{dfn}

\begin{rem} The reasonableness of $\CC$ is roughly what we need to prove that
the comparison to be defined succeeds (it will be used to show that the coarse
tree $\Uu$ that we build does not move fine-structural generators); the
definition is extracted from the proof. The assumption is probably not optimal,
but it seems hard to get by with much less. In typical applications, an
$L[\es,x]$-construction is strongly reasonable or more; the proof that the
comparison succeeds simplifies a little under this extra assumption (but only in
one spot).\end{rem}

\begin{rem}\label{rem:factor}Given an active $N=N_{\alpha+1}$ and
$E^*=E^*_{\alpha+1}$ as in \ref{dfn:construction}, we have a canonical factor
embedding $j:\Ult(N,F^N)\to i_{E^*}(N)$, which is $\Sigma_0$-elementary,
preserves cardinals, and $\crit(j)\geq\nu(F^N)$ and $j\com
i^N_{F^N}=i_{E^*}\rest N$. Using $j$, it's easy to see that if $\CC$ is strongly
reasonable then it is reasonable.\end{rem}

\begin{rem}
 Our main theorem, \ref{thm:main1}, is used in the proofs of \cite[14.3,
16.1]{derived}.
Given a real $x$, part (b) of the conclusion of the theorem can be used
to ensure that for each limit $\lambda<\lh(\Tt)$,
$(x,\Tt\rest\lambda)$ is (class) extender algebra generic over
$M(\Tt\rest\lambda)$. This is used in the proof of \cite[16.1]{derived}. The
first author wishes to thank Nam Trang for pointing out to him that the version
of the theorem given in an earlier draft of the paper, which omitted (b), was
insufficient for the proof of \cite[16.1]{derived}.
In
the construction of $\Tt$ and
$\Uu$, if one omits the use of extenders included specifically for the purposes
of establishing (b), then one still obtains trees satisfying (a) and (c).
The next two definitions relate to part (b).
\end{rem}

\begin{dfn}
A pair $(\Tt,\Uu)$ of padded iteration trees is \emph{neat} iff we have:
(a) $\lh(\Tt)=\lh(\Uu)$; (b) $\Tt$ is on a premouse and is normal; (c) Let
$\lambda\leq\lh(\Tt)$ be a limit. Then either $\Tt\rest\lambda$ is cofinally
non-padded (i.e. $E^\Tt_\alpha\neq\emptyset$ for cofinally many
$\alpha<\lambda$) or $\Uu\rest\lambda$ is cofinally non-padded. If both are
cofinally non-padded then
$\delta(\Tt\rest\lambda)=\delta(\Uu\rest\lambda)$. In any case, let
$\delta_\lambda$ denote $\delta(\Tt\rest\lambda)$ or $\delta(\Uu\rest\lambda)$,
whichever is defined. Then for all limits $\gamma<\lambda\leq\lh(\Tt)$ we
have $\delta_\gamma<\delta_\lambda$.

Assume $(\Tt,\Uu)$ is neat. The \emph{neat code} for $(\Tt,\Uu)$ is the
set of triples $(i,\delta,\gamma)$ such that for some limit $\lambda<\lh(\Tt)$
we have $\delta=\delta_\lambda$ and either (i) $i=0$, $\Tt\rest\lambda$ is
cofinally non-padded and for some
$\alpha\in[0,\lambda)_\Tt$ such that
$E^\Tt_\alpha\neq\emptyset$, we have $\gamma=\lh(E^\Tt_\alpha)$; or (ii) $i=1$
and $\Uu\rest\lambda$ is cofinally non-padded and $\gamma\in[0,\lambda)_\Uu$.
\end{dfn}

\begin{dfn}
Let $M$ be a premouse, let $\sigma<\OR^M$,
$Y\sub\OR$ and $Z\sub\OR^3$. We say that $M$
is \emph{$(Y,Z)$-valid at $\sigma$} iff either $\sigma$ is not a cardinal in
$M$,
or for all $E\in\es_+^M$, if
$\sigma=\nu(E)$ and $M|\lh(E)\sats$``$\sigma$ is inaccessible'' then
$Y\sub\sigma$ and $(Y,Z\inter\sigma^3)$ satisfies all extender
algebra\footnote{Here we mean the ``$\delta$-generator'' extender algebra.
That is, for each $x\in\sigma^{<\om}$ there is a corresponding atomic formula
$\varphi_x$.}
axioms in $M|\lh(E)$
induced by $E$. We say that $M$ is
\emph{$(Y,Z)$-valid} iff $M$ is
$(Y,Z)$-valid at $\sigma$ for every $\sigma<\OR^M$.
\end{dfn}

\begin{dfn}
A coarse iteration tree $\Uu$ is \emph{normalizable} iff it is
nonoverlapping and for each
$\alpha+1<\lh(\Tt)$, $M^\Tt_\alpha\sats$``$E^\Tt_\alpha\in\her_{|V_\rho|+1}$
where $\rho=\rho(E^\Tt_\alpha)$''.\end{dfn}

The key property of a normalizable
tree is the following:

\begin{prop}\label{lem:normalizable} Let $\Uu$ be a normalizable putative tree
on a coarse premouse $R$. Then there is a unique normal padded putative tree
$\Uu'$ on $R$
such that $\lh(\Uu')=\lh(\Uu)$, and for each $\alpha+1<\lh(\Uu)$:
\begin{itemize}
\item[$\bullet$] $E^{\Uu'}_\alpha\neq\emptyset$ iff
$\rho(E^\Uu_\alpha)<\rho(E^\Uu_\beta)$
for all $\beta+1\in(\alpha+1,\lh(\Tt))$,
\item[$\bullet$] if $E^{\Uu'}_\alpha\neq\emptyset$ then
$E^{\Uu'}_\alpha=E^\Uu_\alpha$,
\item[$\bullet$] for all limits $\lambda<\lh(\Uu)$, if $\Uu'$ has non-padded
stages
cofinally in $\lambda$, then $[0,\lambda]_{\Uu'}=[0,\lambda]_\Uu$.
\end{itemize}
\end{prop}
\begin{proof}Omitted.\qed\end{proof}

\begin{thm}[Main Theorem]\label{thm:main1}
Let $\mouseM,\cpmR\in\Her_{\upsilon^+}$ and $x\in\RR\inter\cpmR$. Let
$m,n\leq\om$.

Suppose $\mouseM$ is an $m$-sound, typical\footnote{\emph{Typical} is defined
at the end of \S\ref{sec:intro}.}, normally
$(m,\upsilon^++1)$-iterable
(fine) $x$-premouse. Let $\Sigma_\mouseM$ be an $(m,\upsilon^++1)$-iteration
strategy
for $\mouseM$.

Suppose $\cpmR=(|\cpmR|,\delta^\cpmR)$ is a $(\upsilon^++1)$-iterable coarse
premouse.\footnote{It's not particularly important that $\cpmR$
be a coarse
premouse. We just need that iteration maps on $\cpmR$ are sufficiently
elementary,
for iteration trees using extenders in $V_\conlength^{\cpmR}$ and its images.}
Let $\Sigma_\cpmR$ be a $(\upsilon^++1)$-iteration strategy for $\cpmR$.
Let $\conlength\leq\delta^\cpmR$ and let
$\construction=\left<N_\alpha\right>_{\alpha\leq\conlength}\in\cpmR$ be such
that
$\cpmR\sats$``$\construction$ is a reasonable $L[\es,x]$-construction''.

Let $A\sub\upsilon^+$ with $A$ bounded in $\upsilon^+$.

Then there is
a padded $m$-maximal
normal iteration tree $\Tt$ on $\mouseM$, via
$\Sigma_\mouseM$,
and
a padded iteration tree $\Uu$
on $\cpmR$, via $\Sigma_\cpmR$, both of successor length
$<\upsilon^+$, such
that:
\begin{enumerate}
 \item[(a)] Either:
\begin{itemize}
\item[(i)] $i^\Uu(\core_n(N_\conlength))\ins \I^\Tt$; or
\item[(ii)] $b^\Tt$ does not drop in model or degree and
$\I^\Tt=\core_m(N_\alpha^{i^\Uu(\construction)})$ for some $\alpha\leq
i^\Uu(\conlength)$.
\end{itemize}
 \item[(b)] $(\Tt,\Uu)$
is neat. Let $B$ be the neat code for $(\Tt,\Uu)$.
 \begin{itemize}
  \item[(i)] If $i^\Uu(\core_n(N_\conlength))\pins\I^\Tt$ or
[$b$ drops in model and $i^\Uu(\core_n(N_\conlength))=\I^\Tt$]
then let
$P=i^\Uu(\core_n(N_\conlength))$ and $\rho=\rho_n^P$.
 \item[(ii)] If $\I^\Tt=\core_m(N^{i^\Uu(\construction)}_\alpha)$ for some
$\alpha<i^\Uu(\conlength)$ then let $P=\I^\Tt$ and $\rho=\rho_m^P$.
 \item[(iii)]
Let $k=\min(m,n)$.
If $b$ does not drop in model and
$\I^\Tt=i^\Uu(\core_k(N_\conlength))$ then let $P=\I^\Tt$
and $\rho=\rho_k^P$.
 \end{itemize}
  Let $\tau=(\rho^+)^P$.\footnote{Here if
$\rho=\OR^{P}$ or $\rho$ is the largest
cardinal of $P$ then $(\rho^+)^{P}$ denotes $\OR^{P}$. In particular, if $n=0$
and $P$ is type 3 then $(\rho^+)^P=\OR^P$, not $\OR(\core_0(P))$.}
Then $P|\tau$
is $(A,B)$-valid.
\item[(c)] Let $\CC^\alpha=i^\Uu_{0,\alpha}(\construction)$ for
$\alpha<\lh(\Uu)$.
We may take $\Uu$ to satisfy condition (i) below;
alternatively (ii) below. (But maybe not (i) and
(ii) simultaneously.)\begin{itemize}
\item[(i)] $\Uu$ is nonoverlapping, and for each
$\alpha+1<\lh(\Uu)$, if $E^\Uu_{\alpha}\neq\emptyset$ then
$M^\Uu_\alpha\sats$``There is $\gamma+1\leq
i^\Uu_{0,\alpha}(\conlength)$ such that $N=N^{\CC^\alpha}_{\gamma+1}$ is active
and
$E^\Uu_\alpha=F^*_{\gamma+1}$''; or
\item[(ii)] For each
$\alpha+1<\lh(\Uu)$, if $E^\Uu_{\alpha}\neq\emptyset$ then
$M^\Uu_\alpha\sats$``There is $\gamma+1\leq\conlength$ such that
$N=N^{\CC^\alpha}_{\gamma+1}$ is active, and $E^\Uu_\alpha=E^*_{\gamma+1}$'',
and moreover, if $\cpmR\sats$``$\construction$ is normal'' then $\Uu$ is
normalizable.
\end{itemize}
\end{enumerate}
\end{thm}

It seems we can't strengthen (c)(ii) by replacing ``normalizable'' with
``normal'', since the extraction of a normal tree $\Uu'$ from a normalizable
tree $\Uu$ can change the model of origin for a given extender (e.g. we can have
$E^{\Uu'}_{\alpha'}=E^\Uu_\alpha$ for some $\alpha'<\alpha$, and
$M^{\Uu'}_{\alpha'}=M^\Uu_{\alpha'}\neq M^\Uu_\alpha$), so
(c)(ii)
might fail for $\Uu'$ even if it
held for $\Uu$. Also, conclusion (b) becomes somewhat unclear if we replace $\Uu$ with $\Uu'$ (at least, with regard to the genericity of the code for $\Uu'$).

\begin{proof}[Theorem \ref{thm:main1}]
We will first produce an $m$-maximal normal tree $\Tt$ on $\mouseM$, via
$\Sigma_\mouseM$,
and a tree $\Uu$ on $\cpmR$, via $\Sigma_\cpmR$, each of successor length
$<\upsilon^+$, such that:
\begin{enumerate}[label=(\alph*'),ref=(\alph*')]
 \item\label{i:proof_main1_a} Either:
\begin{itemize}
\item[(i)] $i^\Uu(N_\conlength)\ins \I^\Tt$; or
\item[(ii)] $b^\Tt$ does not drop in model and
$\I^\Tt=N_\alpha^{i^\Uu(\construction)}$ for some $\alpha\leq
i^\Uu(\conlength)$.
\end{itemize}
 \item\label{i:proof_main1_b} $(\Tt,\Uu)$ is neat. Let $B$ be the
neat code for $(\Tt,\Uu)$. If \ref{i:proof_main1_a}(i) holds let
$P=i^\Uu(N_\conlength)$; otherwise let $P=\I^\Tt$. Then $P$ is
$(A,B)$-valid.
 \item\label{i:proof_main1_c} $\Uu$ satisfies \ref{thm:main1}(c)(i)
(alternatively, at our will,
\ref{thm:main1}(c)(ii)).
\end{enumerate}
We will then find $\eps<\lh(\Tt)$ such that $\Tt\rest\eps+1$ and
$\Uu\rest\eps+1$ satisfy the requirements of \ref{thm:main1}.

To construct $\Tt$, we will define a sequence
$\Ttvec=\left<\Tt^\alpha\right>_{\alpha\leq\zeta}$ of padded normal trees on
$\mouseM$, and will set $\Tt=\Tt^\zeta$. The trees $\Tt^\alpha$ approximate
initial
segments of $\Tt$; we will have $\lh(\Tt^\alpha)=\alpha+1$. We simultaneously
construct $\Ttvec$ and $\Uu$, recursively through ordinal stages
$\beta\leq\zeta$. The process
is much like standard comparison, but is also significantly different.

When
beginning stage $\beta$ we will have already built $\Uu\rest\beta$ and
$\Ttvec\rest\beta$. We will then define $\Uu\rest\beta+1$ and
$\Tt^\beta$. For limit $\beta$, the trees $\Ttvec\rest\beta$ will be defined
such
that the sequence converges to a padded tree $\Tt^{<\beta}$ of length $\beta$
with $(\Tt^{<\beta},\Uu\rest\beta)$ neat; we then apply our
iteration strategies to obtain $\Tt^\beta$ and $\Uu\rest\beta+1$. We will show
that for limit $\beta$,
$\Tt^\beta\ins\Tt^\alpha$ for all $\alpha\geq\beta$. (Applying this to the
largest limit $\beta'\leq\alpha$, it follows that
$(\Tt^\alpha,\Uu\rest\alpha+1)$ is neat.) At successor
stages $\alpha+1$, we will choose extenders $E^\Uu_\alpha\in M^\Uu_\alpha$
and
$E=E^{\Tt^{\alpha+1}}_\alpha\in\{\emptyset\}\un\es_+(M^{\Tt^\alpha}_\alpha)$,
such that either $E^\Uu_\alpha\neq\emptyset=
E$ or $E^\Uu_\alpha=\emptyset\neq
E$. If $E\neq\emptyset$ then
we will have $\lh(E)>\lh(F)$ for all extenders $F$ used in $\Tt^\alpha$, and we set
$\Tt^{\alpha+1}=\Tt^\alpha\conc\left<E\right>$, with
$\Tt^{\alpha+1}\pred(\alpha+1)$, etc, determined by $m$-maximality. In this case
we are making a tentative decision to use $E$ in the final tree $\Tt$; this
decision may be tentatively retracted later. If $E=\emptyset$ then we will set
$\Tt^{\alpha+1}=\Tt^\alpha\rest(\gamma+1)\conc\Pp$, where
$\gamma+1\leq\lh(\Tt^\alpha)$ and $\Pp=\left<\emptyset,\emptyset,\ldots\right>$
consists of only padding. Here if $\gamma+1<\lh(\Tt^\alpha)$, we will have
$E'=E^{\Tt^\alpha}_\gamma\neq\emptyset$, and we are tentatively retracting the
use
of $E'$ from the final $\Tt$; we may later change our mind
about $E'$ again. No such retractions occur in the construction of $\Uu$.
Regarding padding, if $E^\Uu_\alpha=\emptyset$ we set
$\Uu\pred(\alpha+1)=\alpha$ and $i^\Uu_{\alpha,\alpha+1}=\id$. If
$E^\Uu_\alpha\neq\emptyset$ then we will ensure that
$E^\Uu_{\Uu\pred(\alpha+1)}\neq\emptyset$ also. Likewise for trees in $\Ttvec$.

We will simultaneously define various other objects in order to guide our
selection of the extenders used in building $\Ttvec,\Uu$, and in order to prove
that the comparison succeeds.

We now begin the construction. We set $\Tt^0$ and $\Uu\rest 1$ to
be the trivial trees on $\mouseM$ and $\cpmR$ respectively.

Now consider stage $\alpha+1$. We are given trees $\Tt^\alpha$ and
$\Uu\rest\alpha+1$, with last models $M^{\Tt^\alpha}_\alpha$ and $M^\Uu_\alpha$
respectively. Define $M^\alpha=M^{\Tt^\alpha}_\alpha$, $R^\alpha=M^\Uu_\alpha$,
$\CC^\alpha=i^\Uu_{0,\alpha}(\construction)$,
$\conlength^\alpha=i^\Uu_{0,\alpha}(\conlength)$
and $N^\alpha_\beta=N^{\CC^\alpha}_\beta$.

We will analyse $M^\alpha$ and $(R^\alpha,\CC^\alpha)$. This will culminate in
either a proof that
our comparison has already succeeded (i.e., $\Tt^\alpha,\Uu\rest\alpha+1$ are as
in \ref{i:proof_main1_a}-\ref{i:proof_main1_c}), or else in a selection of
extenders
$E^{\Tt^{\alpha+1}}_\alpha,E^\Uu_\alpha$,
chosen by finding
the earliest roots of
disagreement between $M^\alpha$ and $\CC^\alpha$, or the first extenders we
reach that, if ignored, could be an obstacle to
validity.
The
analysis is related to resurrection (see \cite[\S 12]{fsit}).
We will in fact define the analysis a little more generally.
After this, we will explain how we determine $\Uu\pred(\alpha+1)$ and any
retraction of extenders required to form
$\Tt^{\alpha+1}$.

\begin{dfn}[$(Y,Z)$-descent]\label{dfn:descent}Let $M$ be an $x$-premouse. Let
$R=(|R|,\delta)$ be a coarse premouse with $x\in R$.
Let $\conlengthabstract\leq\delta$ and let
$\CC=\left<N_\alpha\right>_{\alpha\leq\conlengthabstract}\in R$ be such that
$R\sats$``$\CC$ is a reasonable $L[\es,x]$-construction''.
Let $Y,Z\sub\OR$. The $(Y,Z)$-\emph{descent} of
$(M,(R,\CC))$ is a quadruple $(c,d,e,\theta)$, defined as follows.

We will first define $k<\om$ and
$c = \left<\gamma_i,\xi_i,\mu_i,\theta_i\right>_{i\leq k}$,
with $\gamma_i,\xi_i,\mu_i\in\OR$ for all $i\leq k$, $\theta_i\in\OR$ for
$i<k$, and $\theta_k\in\OR\un\{\dagger\}$. We will also say
``$\theta_k$ is undefined'' to mean ``$\theta_k=\dagger$''.\footnote{Usually
$\theta_k\in\OR$. If in the descent of
$(M^\alpha,(R^\alpha,\CC^\alpha))$
we get $\theta_k=\dagger$ then we will show that the comparison has already been
successful, i.e. $\Tt^\alpha,\Uu\rest\alpha+1$ are as required. Moreover, this
is the only manner in which the comparison can terminate.}

We will have $k\geq 0$. Let $\gamma_0=\OR^M$ and $\xi_0=\conlengthabstract$.

Suppose that for some $i<\om$, we have determined that $k\geq i$, and have
defined $\gamma_i\leq\OR^M$ and $\xi_i\leq\conlengthabstract$.
%

Let $\mu_i$ be the largest ordinal $\mu$ such that ($*$)$_\mu$ holds, where
$(*)_\mu$ asserts:
\begin{enumerate}[label=(\roman*)]
 \item $\mu\leq\gamma_i\inter\OR(N_{\xi_i})$;
 \item $M|\mu=N_{\xi_i}|\mu$;
 \item if $\mu<\gamma_i$ then $\mu$ is a cardinal
of $M|\gamma_i$;
 \item if $\mu<\OR(N_{\xi_i})$ then $\mu$ is a cardinal of
$N_{\xi_i}$;
 \item $M|\mu$ is $(Y,Z)$-valid.
\end{enumerate}

Note $\mu_i$ is well defined, as ($*$)$_0$ holds, and if
$\mu$ is a limit of ordinals $\mu'$ such that ($*$)$_{\mu'}$ holds then both
$M|\mu$, $N_{\xi_i}|\mu$ are passive, and ($*$)$_\mu$ holds.

Let $(\dagger)_i$ be the statement ``$\mu_i=\min(\gamma_i,\OR(N_{\xi_i}))$''.

Suppose $(\dagger)_i$ holds. Then we set $k=i$, stop the analysis, and do
not define $\theta_k$. Note that here if
$\gamma_k<\OR(N_{\xi_k})$ then $\mu_k=\gamma_k$ is a cardinal of $N_{\xi_k}$ and
$M|\gamma_k\pins N_{\xi_k}$, so $M|\gamma_k$ is passive. Likewise if
$\OR(N_{\xi_k})<\gamma_k$.

Now suppose $(\dagger)_i$ fails. So $\mu_i$ is a cardinal of $M|\gamma_i$ and of
$N_{\xi_i}$. Let $\theta_i$ be the sup of all ordinals
$\delta+\om$ such that $\delta\in\gamma_i\inter
N_{\xi_i}$ and $M|\delta=N_{\xi_i}|\delta$ and $M|\delta$ projects to
$\mu_i$
and is $(Y,Z)$-valid at $\mu_i$.

We consider two cases. In the following if $\mu_i$ is the largest cardinal
of $M|\gamma_i$ then $(\mu_i^+)^{M|\gamma_i}$ denotes
$\gamma_i=\OR^{M|\gamma_i}$, and
likewise for $N_{\xi_i}$.

\begin{case}\label{case:terminal}
Either (i) $\theta_i=(\mu_i^+)^{M|\gamma_i}=(\mu_i^+)^{N_{\xi_i}}$,
or (ii) $M|\theta_i=N_{\xi_i}|\theta_i$ is not
$(Y,Z)$-valid.

Then let $k=i$; we have finished defining $c$.
\end{case}

\begin{case}\label{case:nonterminal}
Otherwise.

Then we will have $k>i$; we have not yet finished defining $c$.

If $\theta_i<(\mu_i^+)^{M|\gamma_i}$ then let $\gamma_{i+1}<\gamma_i$ be least
such that $\theta_i\leq\gamma_{i+1}$ and $\rho_\om^{M|\gamma_{i+1}}=\mu_i$.

If $\theta_i=(\mu_i^+)^{M|\gamma_i}$ then let $\gamma_{i+1}=\gamma_i$.

If $\theta_i<(\mu_i^+)^{N_{\xi_i}}$ then let $q\pins N_{\xi_i}$ be least such
that $N_{\xi_i}|\theta_i\ins q$ and $\rho_\om^q=\mu_i$, and let
$\xi_{i+1}<\xi_i$ be
such that $\core_\om(N_{\xi_{i+1}})=\core_0(q)$.

If $\theta_i=(\mu_i^+)^{N_{\xi_i}}$ then let $\xi_{i+1}=\xi_i$.
\end{case}

Suppose Case \ref{case:nonterminal} attains at stage $i$.
Then:
\begin{enumerate}[label=(\alph*),ref=\alph*]
\item\label{item:lexless} $(\gamma_{i+1},\xi_{i+1})<_\lex(\gamma_i,\xi_i)$.
\item $\theta_i\leq\gamma_{i+1}\leq\gamma_i$ and
$M||\theta_i=N_{\xi_i}||\theta_i=N_{\xi_{i+1}}||\theta_i$.
\item Either $\theta_i$ is a cardinal of $M|\gamma_{i+1}$ or
$\theta_i=\gamma_{i+1}$, and
either $\theta_i$ is a cardinal of $N_{\xi_{i+1}}$ or
$\theta_i=\OR(N_{\xi_{i+1}})$.
The latter follows from the universality of standard parameters and condensation
of the models of $\CC$.
\item $\mu_i\leq\mu_{i+1}$.
\item\label{item:mu_i=mu_i+1} Suppose $\mu_i=\mu_{i+1}$. Then
$\mu_{i+1}<\min(\gamma_{i+1},\OR(N_{\xi_{i+1}}))$, $(\dagger)_{i+1}$
fails, $\theta_{i+1}=\theta_i$, and $k=i+1$, but Case
\ref{case:terminal}(ii) fails at stage $i+1$. (If Case
\ref{case:terminal}(ii) attained at stage $i+1$ then it would in fact attain at
stage $i$, by universality.)
\end{enumerate}

There must be a stage $i$ at which Case \ref{case:terminal} attains, by
(\ref{item:lexless}) above. This defines $k$ and $c$. We next define $d$ and
$e$.

Let $\rhovec$ ($\sigmavec$ resp.) be the set of all $\rho$ such that for some
$i<k$,
$\rho=\mu_i$ and $\gamma_{i+1}<\gamma_i$ ($\xi_{i+1}<\xi_i$ resp.). For such
$\rho,i$ with $\rho\in\rhovec$,
let $\gamma_\rho=\gamma_{i+1}$ and $P_\rho=M|\gamma_\rho$
(so $\rho_\om(P_\rho)=\rho$).
For such $\sigma=\rho,i$ with $\sigma\in\sigmavec$,
let $\xi_\sigma=\xi_{i+1}$,
$Q_\sigma=N_{\xi_\sigma}$ and $q_\sigma$ be such that
$\core_0(q_\sigma)=\core_\om(Q_\sigma)$.
Note that
$\rhovec\un\sigmavec=\{\mu_0,\ldots,\mu_{k-1}\}$.
Finally, let
$P_0=M$,
$Q_0=N_\conlengthabstract$, and let $q_0$ be undefined.

Let $d=\left<\gamma_\rho,P_\rho\right>_{\rho\in\{0\}\un\rhovec}$,
$e=\left<q_\sigma,\xi_\sigma,Q_\sigma\right>_{\sigma\in\{0\}\un\sigmavec}$
and $\theta=\theta_k$.
This completes the definition of descent.
\end{dfn}

\begin{rem}\label{rem:distinctions}
We continue with the same notation. Let
$\rho\in\rhovec$ and $\sigma\in\sigmavec$.
We claim that $P_\rho\neq q_\sigma$. For suppose $P_\rho=q_\sigma$.
Let $P=P_\rho$. Then $\rho=\rho_\om^P=\sigma$.
We have $i<k$ such that $\rho=\mu_i$.
Also, $\theta_i=(\rho^+)^P\leq\OR^P$.
But $P\pins P_i$ and $P\pins Q_i$. Therefore $P$ is not $(Y,Z)$-valid at $\rho$, so $P|\theta_i$ is not $(Y,Z)$-valid. But then $k=i$, contradiction.
\end{rem}

\begin{rem}\label{rem:background_exists}
Suppose $(\dagger)_k$ fails. 

We have $P_k=M|\gamma_k$ and $Q_k=N_{\xi_k}$. We have
$M||\theta_k=Q_k||\theta_k$ and this model is $(Y,Z)$-valid. Note that either
$M|\theta_k$ is active or $Q_k|\theta_k$ is active.

Suppose that $M|\theta_k=Q_k|\theta_k$. Let $P=M|\theta_k$. Then $P$
is not $(Y,Z)$-valid. For otherwise, by Case \ref{case:terminal}(i), we have
$(*)_{\theta_k}$, so $\mu_k\geq\theta_k$, contradiction. So $P$ is
type 3, and $\mu_k$ is a limit cardinal of $P_k$ and of
$Q_k$. Moreover, we claim that $P=N_\xi$ for some $\xi$. For suppose $P\pins
Q_k$, and let $\xi$ be such that $\core_\om(N_\xi)=\core_0(P)$.
Then $\rho_\om^P=\mu_k$, and the core embedding $\core_0(P)\to\core_0(N_\xi)$
is in fact the identity. So if $P\neq N_\xi$ then by the initial segment
condition, $P\in N_\xi$, but then by universality, $P\in P$, contradiction.

Now suppose that Case \ref{case:terminal}(i) attains at stage $k$. Then 
if $M|\theta_k$ is active then $\theta_k=\gamma_k$,
and if $Q_k|\theta_k$ is active then $\theta_k=\OR(Q_k)$.

So in any case, $Q_k|\theta_k=N_\xi$ for some $\xi$. (If $Q_k|\theta_k$ is
passive then this is because $\mu_k$ is a cardinal of $Q_k=N_{\xi_k}$ and
$Q_k|\theta_k$ is a limit of levels projecting to $\mu_k$.)
\end{rem}

We now proceed with the construction. Let $B$ be the
neat code for $(\Tt^\alpha,\Uu\rest\alpha+1)$. Consider the
$(A,B)$-descent of $(M,(R,\CC))=(M^\alpha,(R^\alpha,\CC^\alpha))$; we use
notation as in \ref{dfn:descent}.

Suppose that $(\dagger)_k$ fails. Then the comparison has not yet succeeded.
We will specify $E^{\Tt^{\alpha+1}}_\alpha$ and $E^\Uu_\alpha$.
Exactly one of these extenders will be non-empty, with
$E^\Uu_\alpha\neq\emptyset$ if it's reasonable. This helps to organize the
analysis.
\footnote{We might have organized the comparison such that if both $M|\theta_k$
and $N_{\xi_k}|\theta_k$ are active, then
$E^{\Tt^{\alpha+1}}_\alpha=F^{M|\theta_k}$ and
$E^\Uu_\alpha=E^*_{\xi_k}$ (or $E^\Uu_\alpha=F^*_{\xi_k}$).
However, then we may get $M|\theta_k\ins
i_{E^\Uu_\alpha}(N_{\xi_k})$. If
this occurs and $M|\theta_k\neq N_{\xi_k}|\theta_k$ we would want to retract our
use of $F^{M|\theta_k}$ when defining
$\Tt^{\alpha+2}$. This is one motivation to wait longer before using an extender
in $\Ttvec$.}

Let $E=F^{M|\theta_k}$ and $F=F^{N_{\xi_k}|\theta_k}$. If $E\neq\emptyset=F$
then set $E^{\Tt^{\alpha+1}}_\alpha=E$ and
$E^\Uu_\alpha=\emptyset$.
Otherwise $F\neq\emptyset$; in this case set
$E^{\Tt^{\alpha+1}}_\alpha=\emptyset$ (even if $E\neq\emptyset$),
let $\xi$ be
such that $N_{\xi_k}|\theta_k=N_\xi$ (see \ref{rem:background_exists}) and set
$E^\Uu_\alpha$ to be either $E^*=(E^*_{\xi})^\CC$ or
$F^*=(F^*_{\xi})^\CC$,
depending on what properties we want for
$\Uu$. For \ref{thm:main1}(c)(i) we use $F^*$; for \ref{thm:main1}(c)(ii) we use
$E^*$.
In all cases also define $F^\Uu_\alpha=F$.

If $E^\Uu_\alpha\neq\emptyset$ then set $\Uu\pred(\alpha+1)$ to be the least
$\gamma\leq\alpha$ such that $E^\Uu_\gamma\neq\emptyset$ and for all
$\delta\in[\gamma,\alpha)$, $\crit(E^\Uu_\alpha)<\rho(E^\Uu_\delta)$ and
$\crit(E^\Uu_\alpha)\leq\nu(F^\Uu_\delta)$.
Note that if we are following the prescription for \ref{thm:main1}(c)(i), then
we
always have $\rho(E^\Uu_\delta)\leq\nu(F^\Uu_\delta)$, so
$\Uu$ will be non-overlapping.
If
$R\sats$``$\CC$ is normal'' and we are following the prescription for
\ref{thm:main1}(c)(ii),
then $E^*=E^\Uu_\alpha$ is such that $E^*\in\her_{|V_{\rho(E^*)}|+1}$, and for
any $\beta\leq\alpha$, if $\crit(E^\Uu_\alpha)<\rho(E^\Uu_\delta)$ for all
$\delta\in[\beta,\alpha)$, we automatically have
$\crit(E^\Uu_\alpha)\leq\nu(F^\Uu_\beta)$. So in this case, $\Uu$ will be
normalizable.\footnote{If $R\sats$``$\CC$ is not normal'' and we are aiming for
\ref{thm:main1}(c)(ii),
then the clause ``and $\crit(E^\Uu_\alpha)\leq\nu(F^\Uu_\delta)$'' in the
definition
of $\Uu\pred(\alpha+1)$ might prevent $\Uu$ from being nonoverlapping, but it
is needed for our proof to work.}

If $E^{\Tt^{\alpha+1}}_\alpha\neq\emptyset$ then we set
$\Tt^{\alpha+1}=\Tt^\alpha\conc\left<E^{\Tt^{\alpha+1}}_\alpha\right>$;
normality and $m$-maximality determine the remaining structure of
$\Tt^{\alpha+1}$.

Suppose $E^\Uu_\alpha\neq\emptyset$, so $E^{\Tt^{\alpha+1}}_\alpha=\emptyset$.
Suppose there is
$\gamma+1<\lh(\Tt^\alpha)$ such that $E^{\Tt^\alpha}_\gamma\neq\emptyset$, and
$M^\alpha||\lh(E^{\Tt^\alpha}_\gamma)\not\ins
i^{M^\Uu_\alpha}_{E^\Uu_\alpha}(N_{\xi_k})$. Let $\gamma$ be the least such. We
set
$\Tt^{\alpha+1}=(\Tt^\alpha\rest(\gamma+1))\conc\Pp$, where
$\Pp=\left<\emptyset,\emptyset,\ldots\right>$ is only padding, such that
$\lh(\Tt^{\alpha+1})=\alpha+2$. If there is no such $\gamma+1$ we set
$\Tt^{\alpha+1}=\Tt^\alpha\conc\left<\emptyset\right>$.

This completes stage $\alpha+1$ of the comparison, given that $(\dagger)_k$
fails.

\begin{rem}\label{rem:dagger_holds}
Suppose now that $(\dagger)_k$ holds. We
set $\zeta=\alpha$, and claim that the comparison has succeeded, i.e. that
$\Tt=\Tt^\zeta$ and $\Uu=\Uu\rest\zeta+1$ satisfy
\ref{i:proof_main1_a}-\ref{i:proof_main1_c}. We have either
$M|\gamma_k\ins N_{\xi_k}$ or $N_{\xi_k}\ins M|\gamma_k$. First observe that
either $\gamma_k=\gamma_0=\OR^M$ or $\xi_k=\xi_0=\conlengthabstract$. Suppose
not, so $k>0$ and
$\gamma_k<\OR^M$ and $\xi_k<\conlengthabstract$. By \ref{dfn:descent}(e),
$\mu_{k-1}<\mu_k$, so $\rho_\om(M|\gamma_k)<\mu_k$ and
$\rho_\om(N_{\xi_k})<\mu_k$. But $\mu_k$ is a cardinal of both models.
Therefore $M|\gamma_k=N_{\xi_k}$, so
$\mu=\rho_\om(M|\gamma_k)\in\rhovec\inter\sigmavec$ and
$M|\gamma_k=P_\mu=q_\mu$, contradicting \ref{rem:distinctions}.

If $M|\gamma_k\pins N_{\xi_k}$ then $\mu_k=\gamma_k$ (by $(\dagger)_k$) so
$M|\gamma_k$ is a cardinal proper segment of $N_{\xi_k}$. This gives that
$M|\gamma_k=N_\xi$ for some $\xi<\xi_k$, and
$\rho_\om^{M|\gamma_k}=\gamma_k=\gamma_0$, so in fact
$M=N_\xi$,
and $b^\Tt$
does not drop in model or degree. This completes the proof in this case. So
assume $N_{\xi_k}\ins M|\gamma_k$. If $\xi_k=\xi_0$ we are done, and this
follows if $N_{\xi_k}\pins M|\gamma_k$, as in the previous case. So we are
left with the case that $\xi_k<\xi_0$ and $N_{\xi_k}=M|\gamma_k=M$. We must
prove that $b^\Tt$ does not drop in model. We
will do this later, because to do so, and to prove that the comparison
terminates, we first need
to analyse the comparison in detail.
\end{rem}

This completes stage $\alpha+1$ of the comparison.

Given $\left<\Tt^\alpha\right>_{\alpha<\eta}$, $\eta$ a limit, let
$\Tt^{<\eta}=\lim_{\alpha\to\eta}\Tt^\alpha$. That is, $\lh(\Tt^{<\eta})=\eta$
and for all $\gamma<\eta$,
$E^{\Tt^{<\eta}}_\gamma=\lim_{\alpha\to\eta}E^{\Tt^\alpha}_\gamma$. (Note that
the sequence $\left<E^{\Tt^\beta}_\gamma\right>_{\beta\in[\gamma+1,\eta)}$ is of
the form $\Ee\conc\Pp$, where $\Ee=\left<E,E,\ldots\right>$ is constant with
$E\neq\emptyset$ (possibly $\lh(\Ee)=0$), and
$\Pp=\left<\emptyset,\emptyset,\ldots\right>$ (possibly $\lh(\Pp)=0$).) We may
have that $\Tt^{<\eta}$ is eventually only padding, but note that in this case,
$\Uu\rest\eta$ is cofinally non-padded.
Finally, let $\Tt^\eta=\Tt^{<\eta}\conc\Sigma_\mouseM(\Tt^{<\eta})$ and
$\Uu\rest\eta+1=\Uu\rest\eta\conc\Sigma_\cpmR(\Uu\rest\eta)$.

This completes the definition of the comparison.

We now work toward a proof that the comparison succeeds. For this we need to
establish some agreement conditions, by induction through $\lh(\Ttvec,\Uu)$.
First we establish some notation.

Fix $\alpha<\lh(\Ttvec,\Uu)$. With notation as in the definition of the descent
of $(M^\alpha,(R^\alpha,\CC^\alpha))$,
let $c^\alpha=c$, $\gamma^\alpha_i=\gamma_i$, $(\dagger)^\alpha_i=(\dagger)_i$,
etc. Also let $\gamma^\alpha,(\dagger)^\alpha$ denote
$\gamma^\alpha_{k_\alpha},(\dagger)^\alpha_{k_\alpha}$, etc.
If
$(\dagger)^\alpha$ fails and the stage $\alpha$ descent terminates through Case
\ref{case:terminal}(ii), let $P^\alpha_*$ denote
$P^\alpha|\theta^\alpha$. Otherwise let $P^\alpha_*$ denote $P^\alpha$. Define
$Q^\alpha_*$ similarly, and also let $\xi^\alpha_*$ be the $\xi$ such that
$Q^\alpha_*=N^\alpha_\xi$. So if $(\dagger)^\alpha$
fails then
$\mu^\alpha<\theta^\alpha=((\mu^\alpha)^+)^{P^\alpha_*}=((\mu^\alpha)^+)^{
Q^\alpha_* } $ ,
and $P^\alpha_*||\theta^\alpha=Q^\alpha_*||\theta^\alpha$.
Let $\lambda^\alpha$ be the largest
$\lambda\leq\mu^\alpha$ such that $\lambda$ is a limit
of
cardinals of $P^\alpha$ (equivalently, of $Q^\alpha$, $P^\alpha_*$, or
$Q^\alpha_*$).

Let $\eta<\lh(\Ttvec,\Uu)$ be a limit. When $\Tt^{<\eta}$ is cofinally
non-padded, let $M(\Ttvec\rest\eta)$ denote $M(\Tt^{<\eta})$; otherwise, let
$M(\Ttvec\rest\eta)$ denote $M^{\Tt^{\eta}}_\eta|\delta$, where
$\delta=\delta(\Uu\rest\eta)$.\footnote{This might involve a slight abuse of
notation, as $\delta$ need not be determined by $\Ttvec\rest\eta$ alone.}
(These coincide when $\Tt^{<\eta}$ and $\Uu\rest\eta$ are both cofinally
non-padded.)

If the comparison runs to stage $\upsilon^++1$, then we stop it there, producing
$\Tt^{\upsilon^+},\Uu\rest(\upsilon^++1)$; in this case set $\zeta=\upsilon^+$.
Otherwise we stop at the first stage $\zeta+1<\upsilon^+$ such that
$(\dagger)^\zeta$ holds,
producing final trees $\Tt^\zeta,\Uu\rest(\zeta+1)$.

Before beginning the analysis we make a couple more observations.

\begin{rem}\label{rem:E^Tt_eps_neq_emptyset}
Suppose $E^{\Tt^{\eps+1}}_\eps\neq\emptyset$. We have $\lh(E^{\Tt^{\eps+1}}_\eps)=\theta^\eps$,
and the stage $\eps$ descent does not terminate
through Case \ref{case:terminal}(ii). Therefore $\theta^\eps=\gamma^\eps$. Let
$N=M^\eps$. Then $\theta^\eps=\OR^N$ iff $\rhovec^\eps=\emptyset$. Suppose
$\rhovec^\eps\neq\emptyset$, so $k_\eps>0$. Let
$\rhovec^\eps=\{\mu^\eps_{i_0}<\ldots<\mu^\eps_{i_n}\}$, with $i_n<k_\eps$. (If
$\mu^\eps\in\rhovec^\eps$ then $\mu^\eps=\mu^\eps_{k_\eps-1}$.) Then
$\left<\lambda_0,\ldots,\lambda_n\right>=\langle\gamma^\eps_{i_n+1},\ldots,
\gamma^\eps_{i_0+1}\rangle$ is the $\gamma^\eps_{i_n+1}$-model-dropdown sequence
for
$N$ below $\OR^N$.
That is, $\lambda_0=\gamma^\eps_{i_n+1}$ and $\lambda_{i+1}$
is the least $\lambda>\lambda_i$
such that $\lambda<\OR^N$ and
$\rho_\om^{N|\lambda}<\rho_\om^{N|\lambda_i}$, with $n$ as large as possible.
Moreover, $\mu^\eps_{i_j}=\rho_\om(N|\gamma^\eps_{i_j+1})$ for each $j\leq n$.

Similar remarks apply when $E^\Uu_\eps\neq\emptyset$,
but things can be a little different, as it is possible for the stage $\eps$
descent to terminate
through Case \ref{case:terminal}(ii).\end{rem}

\begin{rem}\label{rem:strongly_generators} We will prove that $\Uu$ does not
move fine structural generators. That is, if $\alpha+1<_\Uu\beta+1$ then
$\nu(F^\Uu_\alpha)\leq\crit(E^\Uu_\beta)$. The proof of this depends on other
properties of $\Uu$, to be established inductively, by Claim \ref{clm:agmt}
below. However, if $R\sats$``$\CC$ is strongly reasonable'' then we can prove
the fact right now; the more general case is an elaboration of this argument.
Suppose otherwise. For simplicity, we assume $\Uu$ has no padding. Let
$\beta+1<\lh(\Uu)$ be least such that for some $\alpha+1<_\Uu\beta+1$, we have
$\kappa=\crit(E^\Uu_\beta)<\nu(F^\Uu_\alpha)$. Let $\gamma=\Uu\pred(\beta+1)$.
We claim that $\gamma$ is a successor. For otherwise we have $\alpha+1$ as above
with $\alpha+1<_\Uu\gamma$. By minimality of $\beta+1$,
$\nu(F^\Uu_\alpha)\leq\crit(F^\Uu_{\beta'})$ for all
$\beta'+1\in(\alpha+1,\gamma)_\Uu$. This implies
$\nu(F^\Uu_\alpha)<\rho(E^\Uu_\delta)$ and
$\nu(F^\Uu_\alpha)\leq\nu(F^\Uu_\delta)$ for all
$\delta\in[\alpha+1,\gamma)$. But then $\Uu\pred(\beta+1)<\gamma$,
contradiction. So let
$\alpha+1=\gamma$. By minimality of $\beta+1$, we have
$\kappa<\nu(F^\Uu_\alpha)$. Since $\Uu\pred(\beta+1)>\alpha$,
$\rho(E^\Uu_\alpha)\leq\kappa$. We claim that ($*$) $Q^\alpha_*\sats$``$\kappa$
is
inaccessible''. But then since $M^\Uu_\alpha\sats$``$\CC^\alpha$ is strongly
reasonable'',
$\kappa<\rho(E^\Uu_\alpha)$, contradiction.

So we prove ($*$). We have $\kappa<\rho(E^\Uu_\delta)$ for all
$\delta\in[\alpha+1,\beta)$, so $\kappa$ is measurable in $M^\Uu_{\alpha+1}$,
and therefore inaccessible in
$i^\Uu_{\eps,\alpha+1}(Q^\alpha_*|\mu)=i^{M^\Uu_\alpha}_{E^\Uu_\alpha}
(Q^\alpha_*|\mu)$, where $\mu=\crit(E^\Uu_\alpha)$ and
$\eps=\Uu\pred(\alpha+1)$.
Moreover, $\mu$ is a cardinal of $Q^\alpha_*$, so $\kappa$ is inaccessible in
$U'=i^{M^\Uu_\alpha}_{E^\Uu_\alpha}(Q^\alpha_*)$. But therefore $\kappa$ is
inaccessible in $U=\Ult_0(Q^\alpha_*,F^\Uu_\alpha)$, since
$\kappa<\nu(F^\Uu_\alpha)$ and using the factor embedding $j:U\to U'$ (see
\ref{rem:factor}). Therefore $\kappa$ is inaccessible in $Q^\alpha_*$, as
required.\end{rem}

The following claim lists various facts about the comparison, particularly how
different stages are related. Most of our work will be in giving its statement
and proof. It is proved by induction on $\iota$. Probably the most central fact
is (\ref{a:distinctions}).

\begin{clm}\label{clm:agmt}\textup{
For all $\iota\leq\zeta+1=\lh(\Ttvec,\Uu)$:
\begin{enumerate}[label=(\arabic*),ref=\arabic*]
\item\label{a:<limneat} 
Suppose $\eta\leq\alpha\leq\iota$, $\eta$ is a limit, and either $\alpha$ is a limit or $\alpha<\iota$.
Then $\Tt^\eta=\Tt^\alpha\rest\eta+1$ and
$(\Tt^\alpha,\Uu\rest\alpha+1)$ is neat. Let $B^\alpha$
denote the neat code for $(\Tt^\alpha,\Uu\rest\alpha+1)$. Then
$B^\alpha\sub(\lambda^\alpha+1)^3$. For all $\alpha_1<\alpha_2<\iota$ we have $B^{\alpha_1}=B^{\alpha_2}\inter(\lambda^{\alpha_1}+1)^3$.
\item\label{a:distinctions}
Let $\alpha,\beta<\iota$, let $\rho\in\rhovec^\alpha$ and
$\sigma\in\sigmavec^\beta$. Then $P^\alpha_\rho\neq q^\beta_\sigma$.
\item\label{a:general}
Let $\alpha\leq\beta<\iota$. Then:
\begin{enumerate}[label=(\roman*),ref=\arabic{enumi}\roman*]
\item\label{a:general(ii)} Suppose $\alpha<\beta$. Then
$\lambda^\alpha\leq\lambda^\beta$, the
models $P^\alpha$, $P^\beta$, $Q^\alpha$, $Q^\beta$ agree strictly below
$((\lambda^\alpha)^+)^{P^\alpha_*}=((\lambda^\alpha)^+)^{Q^\alpha_*}$, and
$\lambda^\alpha$ is a limit cardinal of each of these models.
\item\label{a:general(iii)} $\Tt^\alpha$ agrees with $\Tt^\beta$ in terms of use
and indexing of
extenders $E$ such that $\lh(E)<\lambda^\alpha$. That is,
$\Tt^\alpha\rest\gamma+1=\Tt^\beta\rest\gamma+1$ where $\gamma$ is least such
that either $\gamma=\alpha$ or $E^{\Tt^\alpha}_\gamma\neq\emptyset$ and
$\lambda^\alpha<\lh(E^{\Tt^\alpha}_\gamma)$; and if $\gamma\in[\alpha,\beta)$
and $E^{\Tt^\beta}_\gamma\neq\emptyset$ then
$\lambda^\alpha<\lh(E^{\Tt^\beta}_\gamma)$.
\item\label{a:general(iv)} If $\alpha+\om\leq\iota$ then there is $n<\om$ such
that
$\lambda^\alpha<\lambda^{\alpha+n}$.
\end{enumerate}
\item\label{a:Text}
Suppose $\alpha<\beta<\iota$ and $E=E^{\Tt^\beta}_\alpha\neq\emptyset$. Then:
\begin{enumerate}[label=(\roman*),ref=\arabic{enumi}\roman*]
\item\label{a:Text(i)}
$\muvec^{\alpha+1}\inter\lh(E)=\sigmavec^{\alpha+1}\inter\lh(E)$ and
$e^{\alpha+1}\rest\lh(E)=e^\alpha$
\item\label{a:Text(ii)} $\lh(E)\leq\mu^\beta$
\item\label{a:Text(iii)} $\rhovec^\beta\inter\lh(E)=\emptyset$
\item\label{a:Text(iv)} $M^\alpha||\lh(E)\pins P^\beta_*$ and
$M^\alpha||\lh(E)\ins Q^\beta_*$ and $\lh(E)$ is a cardinal of
$P^\beta_*$. If $\lh(E)=\OR(Q^\beta_*)$ then $\beta=\alpha+1$ and
$(\dagger)^\beta$.
\item\label{a:Text(v)} If $F^\Uu_\beta\neq\emptyset$ then $E\rest\nu(E)\not\sub
F^\Uu_\beta$.
\end{enumerate}
\item\label{a:Tretract}
Suppose $\alpha<\beta<\beta+1<\iota$ and $E=E^{\Tt^\beta}_\alpha\neq\emptyset$.
If $E$ is retracted at stage $\beta+1$, i.e.
$E^{\Tt^{\beta+1}}_\alpha=\emptyset$, then $E$ is the last non-empty extender
used in $\Tt^\beta$, $\lh(E)=\mu^\beta=\nu(F^\Uu_\beta)$,
$\lambda^\alpha=\lambda^\beta$ and $\beta<\alpha+\om$.
\item\label{a:Ubelow_crit}
Suppose $\alpha<_\Uu\beta<\iota$ and
$\kappa=\crit(i^\Uu_{\alpha,\beta})<\infty$. Then
$e^\beta\rest\kappa=i^\Uu_{\alpha,\beta}(e^\alpha\rest\kappa)$,
and for all $\rho\in\sigmavec^\alpha\inter\kappa$,
$q^\beta_\rho=q^\alpha_\rho$.
\item\label{a:E^UU}
Suppose $\chi+1<\iota$ and $E^\Uu_\chi\neq\emptyset$. Let
$\alpha=\Uu\pred(\chi+1)$, $\kappa=\crit(E^\Uu_\chi)$ and
$\rho=\max(\{0\}\un(\sigmavec^\alpha\inter\kappa))$. Then:
\begin{enumerate}[label=(\roman*),ref=\arabic{enumi}\roman*]
\item\label{a:E^UU(ii)} $\kappa<\OR(Q^\alpha_\rho)$,
\item\label{a:E^UU(iii)}
$\kappa\leq\lambda^\alpha$
and $Q^\alpha$, $Q^\alpha_\rho$, $Q^\chi$ agree through $\kappa$, which is a
limit cardinal of these models.
\item\label{a:E^UU(iv)} $E^\Uu_\chi$ does not move fine structural generators.
That is, given any $\gamma+1<_\Uu\chi+1$ such that $E^\Uu_\gamma\neq\emptyset$,
we have $\nu(F^\Uu_\gamma)\leq\crit(E^\Uu_\chi)$.
\end{enumerate}
Let $\mu'$ be the largest cardinal $\mu''$ of
$N=i^{M^\Uu_\chi}_{E^\Uu_\chi}(Q^\chi_*)$ such that $\mu''\leq\mu^\chi$. Let
$\theta'=((\mu')^+)^{Q^\chi_*}$.
Then:
\begin{enumerate}[resume*]
\item\label{a:E^UU(v)} $N||\theta'=Q^\chi||\theta'$.
\item\label{a:E^UU(vi)} Suppose $\mu'<\mu^\chi$. Then $F^\Uu_\chi$ is type 1 or
type
3 and $\theta'=\mu^\chi$.
\item\label{a:E^UU(vii)} $Q^{\chi+1}_\rho$ and $N$ agree through
$i^\Uu_{\alpha,\chi+1}(\kappa)=i^{M^\Uu_\chi}_{E^\Uu_\chi}(\kappa)$, a
cardinal
of these models,
\item\label{a:E^UU(viii)} $\lambda^\chi\leq\mu'\leq\mu^{\chi+1}$ and  $\mu'$ is
a
cardinal of $Q^{\chi+1}_\rho$ and of $Q^{\chi+1}$,
\item\label{a:E^UU(ix)} $\mu'\leq\min(\sigmavec^{\chi+1}\inter[\kappa,\infty))$,
\item\label{a:E^UU(x)} The models
$Q^{\chi+1}_*$,
$Q^{\chi+1}$, $Q^{\chi+1}_\rho$, $Q^\chi$ agree strictly below $\theta'$.
\end{enumerate}
\item\label{a:E^UT}
Suppose $\chi+1<\iota$ and $E^\Uu_\chi\neq\emptyset$. We use the notation of
(\ref{a:E^UU}) and let $\mu=\mu^\chi$. If there is no retraction at stage
$\chi+1$, i.e. if $\Tt^{\chi+1}=\Tt^\chi\conc\left<\emptyset\right>$, then let
$\beta=\chi$. If there is retraction, let $\beta=\gamma$, where
$E^{\Tt^\chi}_\gamma$ is the retracted extender. So
$M^{\chi+1}=M^\beta$.
\begin{enumerate}[label=(\roman*),ref=\arabic{enumi}\roman*]
\item\label{a:E^UT(i)} $\rhovec^{\chi+1}\inter\mu'=\rhovec^\beta\inter\mu'$ and
$P^{\chi+1}_{\rho'}=P^\beta_{\rho'}$ for all
$\rho'\in\{0\}\un(\rhovec^{\chi+1}\inter\mu')$.
\item\label{a:E^UT(ii)} Let $\rho'=\max(\{0\}\un(\rhovec^{\chi+1}\inter\mu'))$.
Then $\mu'$
is a cardinal of $P^{\chi+1}_{\rho'}$.
\end{enumerate}
\item\label{a:<limUpad}
Suppose $\eta\leq\iota$ is a limit such that $\Uu\rest\eta$ is eventually only
padding. So
$M^\Uu_{\eta}=\lim_{\beta<\eta}M^\Uu_\beta$. Let
$\delta=\delta(\Ttvec\rest\eta)$. Let $(c,d,e,\theta)$ be
the
$(A,B^\eta)$-descent of $(M(\Ttvec\rest\eta),(M^\Uu_{\eta},\CC^\eta))$. Let
$\delta=\delta(\Ttvec\rest\eta)$. Then:
\begin{enumerate}[label=(\roman*),ref=\arabic{enumi}\roman*]
\item\label{a:<limUpad(i)} If $\gamma\leq\beta<\eta$ are such that
$\Uu\rest\eta=\Uu\rest(\gamma+1)\conc\Pp$, then $e^\gamma\sub e^\beta$.
\item\label{a:<limUpad(ii)} $e=e^\alpha$ for all sufficiently large
$\alpha<\eta$.
Let $\alpha<\eta$ be such that $\Uu\rest\eta=\Uu\rest(\alpha+1)\conc\Pp$ and
$e^\alpha=e$. Let
$\rho=\max(\{0\}\un\sigmavec^\alpha)$. Then $\Tt^{<\eta}\rest[\alpha,\eta)$ is
given by the standard comparison of the phalanx $\Phi(\Tt^\alpha)$ with
$Q^\alpha_\rho=Q^e_\rho$.
\item\label{a:<limUpad(iii)} $\delta$ is a limit of cardinals of $Q^e_\rho$.
\end{enumerate}
\item\label{a:limUpad}
Assume the hypotheses and notation of (\ref{a:<limUpad}) and also that
$\eta<\iota$. Then:
\begin{enumerate}[label=(\roman*),ref=\arabic{enumi}\roman*]
\item\label{a:limUpad(i)} $e^\eta\rest\delta=e$, so $Q^\eta_\rho=Q^e_\rho$;
\item\label{a:limUpad(ii)} $\rhovec^\eta\inter\delta=\emptyset$;
\item\label{a:limUpad(iii)} $\delta$ is a limit of cardinals of
$M^\eta$,
$Q^\eta_\rho$, $Q^\eta$, these models agree through $\delta$, and
$\delta\leq\lambda^\eta$;
\item\label{a:limUpad(iv)} If
$\delta=\min(\OR(Q^\alpha_\rho),\OR(M^\eta))$ then
$(\dagger)^\eta$.
\end{enumerate}
\item\label{a:<limUnonpad}
Suppose $\eta\leq\iota$ is a limit and $\Uu\rest\eta$ is cofinally non-padded.
Let
\begin{equation}\label{eqn:limitsigma}
\sigmavec^{<\eta}=\bigcup_{\xi<_\Uu\eta}\sigmavec^\xi\inter\crit(i^\Uu_{\xi,\eta
}).
\end{equation}
Then there is $\xi<_\Uu\eta$ such that
$\sigmavec^{<\eta}=\sigmavec^\xi\inter\crit(i^\Uu_{\xi,\eta})$. Let $\xi$ be
such, let $\rho=\max(\{0\}\un\sigmavec^{<\eta})$ and
$\delta=\delta(\Uu\rest\eta)$. Then
$M(\Ttvec\rest\eta)=i^\Uu_{\xi,\eta}(Q^\xi_\rho)|\delta$ and $\delta$ is a limit
of cardinals of $i^\Uu_{\xi,\eta}(Q^\xi_\rho)$.
\item\label{a:limUnonpad}
Assume the hypotheses of (\ref{a:<limUnonpad}) and also that $\eta<\iota$. Fix
$\delta,\xi,\rho$ as there.
Then:
\begin{enumerate}[label=(\roman*),ref=\arabic{enumi}\roman*]
\item\label{a:limUnonpad(i)} $\sigmavec^\eta\inter\delta=\sigmavec^{<\eta}$ and
$e^\eta\rest\delta=i^\Uu_{\xi,\eta}(e^\xi\rest\kappa)$ where
$\kappa=\crit(i^\Uu_{\xi,\eta})$.
\item\label{a:limUnonpad(ii)} $\delta\leq\OR(Q^\eta_\rho)$, $\delta$ is a limit
of cardinals of
$Q^\eta_\rho$ and $\delta\leq\lambda^\eta$. If $\delta=\OR(Q^\eta_\rho)$ then
$(\dagger)^\eta$.
\item\label{a:limUnonpad(iii)}
$Q^\eta|\delta=Q^\eta_\rho|\delta=M(\Ttvec\rest\eta)=\liminf_{\gamma<\eta}
Q^\gamma$.
\item\label{a:limUnonpad(iv)} Suppose $\Tt^{<\eta}$ is cofinally non-padded.
Then $\delta$ is a limit of
cardinals of $M^\eta$, and $\rhovec^\eta\inter\delta=\emptyset$. If
$\delta=\OR(M^\eta)$ then $(\dagger)^\eta$.
\item\label{a:limUnonpad(v)}
Suppose $\Tt^{<\eta}$ is eventually only padding. Then there is $\gamma<\eta$
such that: $\Tt^{<\eta}=\Tt^\gamma\conc\Pp$ (so
$M^\eta=M^\gamma$),
$d^\gamma\rest\lambda^\gamma=d^\eta\rest\delta$, and letting
$\rho'=\max(\{0\}\un(\rhovec^\gamma\inter\lambda^\gamma))$, $\delta$ is a limit
of cardinals of $P^\gamma_{\rho'}=P^\eta_{\rho'}$ and of $P^\eta$. Moreover,
$M(\Ttvec\rest\eta)\ins P^\eta_{\rho'}$. If $\delta=\gamma^\eta_{\rho'}$ then
$(\dagger)^\eta$.
\end{enumerate}
\end{enumerate}
}
\setcounter{case}{0}
\end{clm}
\begin{proof}
We proceed by induction on $\iota$. We write, for example,
``(\ref{a:general(ii)})($\iota<5$)'' for (\ref{a:general(ii)}) for values of
$\iota<5$.

\begin{case} $\iota\leq 1$\end{case}
For $\iota=0$ the claim is trivial. For $\iota=1$ the only non-trivial item
is (\ref{a:distinctions}), which follows \ref{rem:distinctions}.

\begin{case} $\iota=\chi+2$\end{case}
We
must prove (\ref{a:<limneat}) for $\alpha=\chi+1$,
(\ref{a:distinctions}) for $\max(\alpha,\beta)=\chi+1$,
(\ref{a:general}),(\ref{a:Text}),(\ref{a:Ubelow_crit}) for $\beta=\chi+1$,
(\ref{a:Tretract}) for $\beta+1=\chi+1$, and (\ref{a:E^UU}),(\ref{a:E^UT}) for
$\chi+1$.

(\ref{a:E^UU(ii)}),(\ref{a:E^UU(iii)}): We have that $\kappa$ is a limit
cardinal of $M^\Uu_\alpha$ and
$Q^\alpha,Q^\alpha_\rho\in M^\Uu_\alpha$, and
$\kappa\leq\nu(F^\Uu_\alpha)<\OR(Q^\alpha_*)$, so $\kappa\leq\lambda^\alpha$,
and by choice of $\rho$, $Q^\alpha|\kappa=Q^\alpha_\rho|\kappa$ and
$\kappa<\OR(Q^\alpha_\rho)$. Applying (\ref{a:general(ii)}) to $\alpha,\chi$, we
get that $Q^\alpha|\kappa=Q^\chi|\kappa$ and $\kappa$ is a limit of cardinals of
$Q^\chi$; in fact $\kappa<\OR(Q^\chi_*)$ since $\kappa=\crit(F^\Uu_\chi)$.

(\ref{a:E^UU(iv)}): Let
$\gamma+1\leq_\Uu\alpha$ with $E^\Uu_\gamma\neq\emptyset$. Suppose that
$\kappa=\crit(E^\Uu_\chi)<\nu(F^\Uu_\gamma)$.

By (\ref{a:E^UU(iv)})($\iota=\chi+1$) and part of
\ref{rem:strongly_generators}, we may assume $\gamma+1=\alpha$, and
$\rho(E^\Uu_\gamma)\leq\kappa$. This will lead to a contradiction with the
reasonableness of $\CC^\gamma$ in $M^\Uu_\gamma$.

We need to establish the hypotheses on $\kappa$ given in
\ref{dfn:construction}(f). We will first establish the appropriate facts about
$U=\Ult(M^\Uu_\gamma,E^\Uu_\gamma)$, and then if $E^\Uu_\gamma\neq
E^*=(E^*_{\xi^\gamma_*})^{\CC^\gamma}$, deduce them about
$U'=\Ult(M^\Uu_\gamma,E^*)$.

As in \ref{rem:strongly_generators}, $\kappa$ is measurable in
$M^\Uu_{\gamma+1}$ and so in $\Ult(M^\Uu_\gamma,E^\Uu_\gamma)$. Let
$\xi<\eta=((\lambda^\gamma)^+)^{Q^\gamma_*}$. By
(\ref{a:general(ii)})($\iota=\chi+1$),
$Q^\chi||\eta=Q^\gamma||\eta=Q^\gamma_*||\eta$ and
$\lambda^\gamma$ is a limit cardinal of $Q^\chi_*$. If
$\lambda^\gamma=\nu(F^\Uu_\chi)$ let $F=F^\Uu_\chi$ and
$E=E^\Uu_\chi\rest\lambda^\gamma$. If $\lambda^\gamma<\nu(F^\Uu_\chi)$ then
there is
$\varsigma<\eta$
such that $\xi<\varsigma$ and
$F=F^\Uu_\chi\rest\varsigma+1$ is non-type
Z.
Then $F$ and $E=E^\Uu_\chi\rest\varsigma+1$ are both generated by
$\lambda^\gamma\un\{\varsigma\}$.
Moreover, by the initial segment condition, there is $\delta$ with
$F=F^{Q^\chi_*|\delta}$. Moreover, letting $N=Q^\chi_*|\delta$, either $N\pins
Q^\gamma||\eta$ or $N||\eta=Q^\gamma||\eta$.

Now we claim that $N,E\in\Ult(M^\Uu_\gamma,F^\Uu_\gamma)$ and $E$ is an extender
there. If $\kappa<\lambda^\gamma$ let $i=1$; if $\kappa=\lambda^\gamma$ let
$i=2$. Then $(N,E)$ is coded by an element of
\begin{equation}\label{eqn:Vagmt}V_{\lambda^\gamma+i}^{M^\Uu_\chi}\sub
V_{\lambda^\gamma+i}^{M^\Uu_{\gamma+1}}=
V_{\lambda^\gamma+i}^{\Ult(M^\Uu_\gamma,E^\Uu_\gamma)}. \end{equation}

To see line (\ref{eqn:Vagmt}), suppose first $i=1$. Then for every
$\delta\in[\gamma,\chi)$,
$\lambda^\gamma\leq\lambda^\delta\leq\nu(F^\Uu_\delta)$, and
$F^\Uu_\delta\rest\nu(F^\Uu_\delta)\sub E^\Uu_\delta$, so
$i_{E^\Uu_\delta}(\crit(E^\Uu_\delta))\geq\lambda^\gamma$. This gives
(\ref{eqn:Vagmt}) in this case. If $i=2$ then
$\lambda^\gamma<i_{E^\Uu_\delta}(\crit(E^\Uu_\delta))$ for every such $\delta$,
which suffices. Now in either case,
$V_{\kappa+1}(M^\Uu_{\gamma+1})=V_{\kappa+1}(M^\Uu_\chi)$, so $E$ is an extender
in $M^\Uu_{\gamma+1}$ and in $\Ult(M^\Uu_\gamma,F^\Uu_\gamma)$, as required.

Finally, suppose $E^*\neq E^\Uu_\gamma$. So we are following the prescription
for \ref{thm:main1}(c)(i), and $E^\Uu_\gamma=E^*\rest\beta$ for some
$\beta\geq\nu(F^\Uu_\gamma)$. So we have a fully elementary $j:U\to
U'=\Ult(M^\Uu_\gamma,E^*)$ with $\crit(j)\geq\beta$. So $j(\kappa)=\kappa$ and
$\kappa$ is measurable in $U'$. Moreover, fixing $\xi,N,E$ as above, $N'=j(N)$
and $E'=j(E)$ witness
\ref{dfn:construction}(f) with respect to $\xi$.

Now since $M^\Uu_\gamma\sats$``$\CC^\gamma$ is reasonable'', we have that
$\kappa<\rho(E^*)$, so $\kappa<\rho(E^\Uu_\gamma)$, contradiction.

(\ref{a:E^UU(v)}),(\ref{a:E^UU(vi)}): Let
$i=i^{M^\Uu_\chi}_{E^\Uu_\chi}$ and $j:\Ult_0(Q^\chi_*,F^\Uu_\chi)\to
N=i(Q^\chi_*)$
the factor map. 
Now $\mu^\chi$ is the largest cardinal of $Q^\chi_*$, so is a cardinal of
$\Ult_0(Q^\chi_*,F^\Uu_\chi)$.
So if $\crit(j)\geq\lh(F^\Uu_\chi)$ or $\mu^\chi$ is a limit cardinal of
$Q^\chi_*$ then $\mu'=\mu^\chi$, and condensation (and that $Q^\chi_*\ins
Q^\chi$) gives (\ref{a:E^UU(v)}). Suppose
$\crit(j)=\mu^\chi$ is a successor cardinal of $Q^\chi$. Then
$Q^\chi_*=Q^\chi$ and $\mu^\chi$ is not
a cardinal of $N$,
so $\mu'<\mu^\chi$ and $\mu'$ is the largest cardinal of $N$ which is
$<\mu^\chi$. Since $\crit(j)\geq\nu(F^\Uu_\chi)$, $F^\Uu_\chi$ is
either type 1 or type 3.
Moreover, $Q^\chi||\OR(Q^\chi)=U|\OR(Q^\chi)$ and $U|\mu^\chi=N||\mu^\chi$,
so $\mu'$ is the largest cardinal of $Q^\chi$ which is $<\mu^\chi$, so
$\mu^\chi=((\mu')^+)^{Q^\chi}=\theta'$.

(\ref{a:Tretract}): Assume $\beta+1=\chi+1$ and $\alpha<\chi$ is such that
$E=E^{\Tt^\chi}_\alpha\neq\emptyset$ is retracted at stage $\chi+1$, i.e.
$E^{\Tt^{\chi+1}}_\alpha=\emptyset$. We use (\ref{a:Text})($\iota=\chi+1$). By
(\ref{a:Text(ii)}), $\lh(E)\leq\mu^\chi$. By (\ref{a:E^UU(v)}),
with $N$ as there, $M^\chi||\mu^\chi=N||\mu^\chi$. But $N|\lh(E)\neq
M^\alpha||\lh(E)=M^\chi|\lh(E)$, since $E$ is being
retracted. So $\lh(E)=\mu^\chi$ and $N|\mu^\chi$ is active. By
(\ref{a:E^UU(v)}), $\mu'<\mu^\chi$, so $\mu^\alpha=\mu'$, and by
(\ref{a:Text(ii)}), $E$ is the last extender used in $\Tt^\chi$. By
(\ref{a:E^UU(vi)}), $\mu'=\nu(F^\Uu_\chi)$.

The fact that $\lambda^\chi=\lambda^\alpha$ follows by
(\ref{a:Text(iv)}) and since $\mu^\chi=\lh(E)$. By
(\ref{a:general(iv)})($\iota=\chi$) this implies $\chi<\alpha+\om$.

(\ref{a:<limneat}): We may assume that $\alpha=\chi+1$ and $\eta$ is the largest limit $\leq\chi$. Let $\delta=\delta(\Tt^\eta\rest\eta)$. Then $\delta\leq\lambda^\eta\leq\lambda^\chi$ by (\ref{a:limUnonpad(ii)}),(\ref{a:limUpad(iii)}),(\ref{a:general})($\iota=\chi+1$). So the property follows from (\ref{a:Tretract}).

(\ref{a:Ubelow_crit}): We may assume $\beta=\chi+1$ and
$\alpha=\Uu\pred(\chi+1)$.

It suffices to prove ($*$)
$\sigmavec^{\chi+1}\inter\kappa=\sigmavec^\alpha\inter\kappa$ and for all
$\rho\in\sigmavec^{\chi+1}\inter\kappa$,
$\xi^{\chi+1}_\rho=i^\Uu_{\alpha,\chi+1}(\xi^\alpha_\rho$); and letting
$\gamma\leq\chi$ be such that $\Tt^{\chi+1}=\Tt^\gamma\conc\Pp$, we have
$M^{\chi+1}=M^\gamma$,
$\rhovec^{\chi+1}\inter\kappa=\rhovec^\gamma\inter\kappa$, and for all
$\rho\in\rhovec^{\chi+1}\inter\kappa$,
$\gamma^{\chi+1}_\rho=\gamma^\gamma_\rho$.

Since $\kappa<\mu^\chi$, by (\ref{a:Tretract}) and
(\ref{a:general(ii)})($\iota=\chi+1$): ignoring padding, either
$\Tt^{\chi+1}=\Tt^\chi$ or $\Tt^{\chi+1}\conc\left<E\right>=\Tt^\chi$ for some
$E$ such that $\kappa<\lh(E)$; and $\kappa\leq\lambda^\gamma,\lambda^\alpha$,
and
$P^\gamma|\kappa=Q^\chi|\kappa=Q^\alpha|\kappa$ and $\kappa$ is a cardinal of
$P^\gamma$, $Q^\chi$ and $Q^\alpha$. Therefore also
$M^{\chi+1}|\kappa=M^\chi|\kappa=Q^\chi|\kappa$.
Moreover, by (\ref{a:<limneat}),
$B^{\chi+1}$, $B^\chi$, $B^\alpha$ and $B^\gamma$ have the same intersection
with
$\kappa^3$,
and $Q^\chi|\kappa$ is $(A,B^{\chi})$-valid.

Since $\kappa<\OR(P^\gamma)$, for all $\rho\in\rhovec^\gamma$, we have
$\kappa<\gamma^\gamma_\rho$. Now by (\ref{a:distinctions})($\iota=\chi+1$), for
all $\rho\in\rhovec^\gamma\inter\sigmavec^\alpha$, $P^\gamma_\rho\neq
q^\alpha_\rho$. This will give the claim, by induction through
$(d,e)^{\chi+1}\rest\kappa$. That is, we have $P^{\chi+1}_0=P^\gamma_0$
and $\xi^{\chi+1}_0=i^\Uu_{\alpha,\chi+1}(\xi^\alpha_0)$. Since
$\kappa=\crit(E^\Uu_\chi)$, $Q^{\chi+1}_0$ and $Q^\alpha_0$ agree below
$\kappa$, and have the same cardinals below $\kappa$.
Assume
$\left<\mu_0,\ldots,
\mu_j\right>=(\sigmavec^\alpha\un\rhovec^\gamma)\inter\kappa\neq\emptyset$; the
contrary case is simpler.
Note that $\mu_0<\kappa$ is a cardinal in both $P^{\chi+1}_0$ and
$Q^{\chi+1}_0$. If $\mu_0\in\rhovec^\gamma\cut\sigmavec^\alpha$ then
$Q^{\chi+1}_0$, $Q^\alpha$, $Q^\chi$, $P^\chi$ and $P^\gamma_{\mu_0}$ agree
beyond their common value $\mu^*$ for $\mu_0^+$,
and $Q^\alpha|\mu^*$ is
$(A,B^\alpha)$-valid,
but $P^\gamma_{\mu_0}\pins P^\gamma_0$
and $P^\gamma_{\mu_0}$ projects to $\mu_0$. So
$\mu^{\chi+1}_0=\mu_0\in\rhovec^{\chi+1}\cut\sigmavec^{\chi+1}$ and
$\gamma^{\chi+1}_1=\gamma^{\chi+1}_{\mu_0}=\gamma^\gamma_{\mu_0}$, and
$\xi^{\chi+1}_1=\xi^{\chi+1}_0$. If $\mu_0\in\sigmavec^\alpha\cut\rhovec^\gamma$
it's similar, noting that $q^\alpha_{\mu_0}\pins Q^\alpha_0|\kappa$ because
$\kappa$ is a cardinal of $M^\Uu_\alpha$, so
$q^{\chi+1}_{\mu_0}=q^\alpha_{\mu_0}$. If
$\mu_0\in\rhovec^\gamma\inter\sigmavec^\alpha$ then use that
$P^\gamma_{\mu_0}\neq q^\alpha_{\mu_0}$. Now iterate this argument through to
$\mu_j$,
resulting in $\gamma^{\chi+1}_{j+1}=\gamma^\gamma_\rho$ where
$\rho=\max(\{0\}\un(\rhovec^\gamma\inter\kappa))$, and
$\xi^{\chi+1}_{j+1}=i^\Uu_{\alpha,\chi+1}(\xi^\alpha_\sigma)$ where
$\sigma=\max(\{0\}\un(\sigmavec^\alpha\inter\kappa))$. Then
$P^{\chi+1}_{j+1}|\kappa=Q^\chi|\kappa=Q^{\chi+1}_{j+1}|\kappa$, and $\kappa$ is
a cardinal of both $P^{\chi+1}_{j+1}$ and $Q^{\chi+1}_{j+1}$, so
$\kappa\leq\mu^{\chi+1}_{j+1}$. This proves ($*$).

(\ref{a:E^UU(vii)}): This follows from the fact that
$Q^{\chi+1}_\rho=i^\Uu_{\alpha,\chi+1}(Q^\alpha_\rho)$ (just shown), and that
$Q^\alpha_\rho|\kappa=Q^\chi_*|\kappa$, and $\kappa$ is a cardinal of $Q^\chi_*$
and
of $Q^\alpha_\rho$. 

(\ref{a:E^UU(viii)})-(\ref{a:E^UU(x)}); (\ref{a:E^UT}): By
(\ref{a:E^UU(v)}),(\ref{a:E^UU(vii)}), we have that $\mu'$ is a cardinal of
$Q^{\chi+1}_\rho$ and $Q^{\chi+1}_\rho||\theta'=Q^\chi||\theta'$ (this gives
part
of (\ref{a:E^UU(x)})).
We proved, in the argument for (\ref{a:Ubelow_crit}), that (\ref{a:E^UT(i)})
holds with ``$\kappa$'' replacing ``$\mu'$''. If there is no retraction things
are easier; assume otherwise, so $\gamma<\chi$, and by (\ref{a:Tretract}),
$\lh(E)=\mu^\chi$ where $E=E^{\Tt^\chi}_\gamma$, and
$\mu^\gamma=\mu'<\mu^\chi=\theta'$. Let
$\rho'=\max(\{0\}\un(\rhovec^\gamma\inter\mu'))$. So $E$ is active on
$P^\gamma_*\ins P^\gamma_{\rho'}$ and $\mu'$ is a cardinal of $P^\gamma_{\rho'}$
(which will give (\ref{a:E^UT(ii)})). Moreover, $P^\gamma_{\rho'}||\lh(E)$ is a
cardinal segment of $Q^\chi$, so
$P^\gamma_{\rho'}||\mu^\chi=Q^{\chi+1}_\rho||\mu^\chi$ (which will give
(\ref{a:E^UU(x)})).
Also, $B^{\chi+1}=B^\chi$ and $Q^\chi|\mu^\chi$ is $(A,B^\chi)$-valid.
Now an induction through $(d,e)^{\chi+1}\rest\mu'$ like for
(\ref{a:Ubelow_crit}) gives (\ref{a:E^UT(i)}) and (\ref{a:E^UU(ix)}),
and the observations above give (\ref{a:E^UT(ii)}) and
(\ref{a:E^UU(viii)}),(\ref{a:E^UU(x)}).

(\ref{a:Text(i)}): Assume $\alpha=\chi$ for non-triviality. An argument like for
(\ref{a:Ubelow_crit}) works, using the facts: $B^{\chi+1}=B^\chi$; 
and $E^\Uu_\chi=\emptyset$, so $M^\Uu_{\chi+1}=M^\Uu_\chi$ and
$\CC^{\chi+1}=\CC^\chi$; and
$E=E^{\Tt^{\chi+1}}_\chi\neq\emptyset$, $\theta^\chi=\lh(E)$ is a
cardinal of
$N=M^{\chi+1}$, and $N|\theta^\chi=Q^\chi|\theta^\chi$ is
passive.

(\ref{a:Text(ii)})-(\ref{a:Text(iv)}): If $\alpha=\chi$, use (\ref{a:Text(i)})
and the facts above.
Suppose $\alpha<\chi$. If $E^{\Tt^{\chi+1}}_\chi\neq\emptyset$ then
$\lh(E^{\Tt^{\chi+1}}_\chi)>\lh(E)$, and (\ref{a:Text})($\iota=\chi+1$) implies
the result. Suppose instead $E^\Uu_\chi\neq\emptyset$. By
(\ref{a:Text(ii)})($\iota=\chi+1$), $\lh(E)\leq\mu^\chi$.
By (\ref{a:E^UU}), defining $\mu',\rho$ as there, $\mu'\leq\mu^{\chi+1}$,
$Q^{\chi+1}_\rho|\mu'=Q^\chi|\mu'$ and $\mu'$ is a cardinal of $Q^{\chi+1}_\rho$
and $Q^\chi$. Let $N=M^{\chi+1}$. Since $\Tt^{\chi+1}$ uses $E$,
$\lh(E)$ is a cardinal of $N$. By (\ref{a:Text(iv)})($\iota=\chi+1$),
$N|\lh(E)\pins Q^\chi_*$ and $\lh(E)$ is a cardinal of $Q^\chi$. If
$\lh(E)\leq\mu'$ this suffices. Assume $\mu'<\lh(E)=\mu^\chi$. Since $E$ was not
retracted, by (\ref{a:E^UU}),
$Q^{\chi+1}_\rho|\lh(E)=N|\lh(E)$, these are passive, but since $\mu'<\lh(E)$,
$\lh(E)$ is not a cardinal of $Q^{\chi+1}_\rho$. Therefore
$\mu'\in\sigmavec^{\chi+1}\cut\rhovec^{\chi+1}$, and
$((\mu')^+)^{q^{\chi+1}_{\mu'}}=\lh(E)\in q^{\chi+1}_{\mu'}$, which gives the
result.

(\ref{a:Text(v)}): This follows from (\ref{a:Text(iv)}) and the initial segment
condition.

(\ref{a:general}): For (\ref{a:general(ii)})
we may assume $\alpha=\chi$; use
(\ref{a:Text(ii)}),(\ref{a:Text(iv)}) and
(\ref{a:E^UU(viii)}),(\ref{a:E^UU(x)}). For
(\ref{a:general(iii)}) we may assume $\alpha=\chi$. If
$E=E^{\Tt^{\chi+1}}_\chi\neq\emptyset$ then
$\Tt^{\chi+1}=\Tt^\chi\conc\left<E\right>$ and $\lambda^\chi<\lh(E)$. If $E$ is
an extender retracted at $\chi+1$ then by (\ref{a:Tretract}),
$\lambda^\chi<\lh(E)$. Part (\ref{a:general(iv)}) is trivial by induction.

(\ref{a:distinctions}): Suppose otherwise. We may assume
$\max(\alpha,\beta)=\chi+1$. By \ref{rem:distinctions}, $\alpha\neq\beta$. Let
$P=P^\alpha_\rho=q^\beta_\sigma$. We have
$\rho=\rho_\om^P=\sigma$.
Let
$(\Tt^\alpha)',(\Tt^\beta)'$ be the non-padded trees equivalent to
$\Tt^\alpha,\Tt^\beta$. We claim that $(\Tt^\alpha)'=(\Tt^\beta)'\rest\gamma+1$
for some $\gamma+1<\lh((\Tt^\beta)')$, and in fact $\gamma$ is least such that
$(\rho^+)^P\leq\lh(E^{(\Tt^\beta)'}_\gamma)\leq\OR^P$.

To see this, first note that $P\pins M^\alpha$,
and
$(\Tt^\alpha)'$ is
$m$-maximal and via $\Sigma_\mouseM$, and by (\ref{a:Text(iii)}),
$\lh(E)\leq\rho$ for
each $E$ used by $(\Tt^\alpha)'$, and $P|\rho\ins\I^{(\Tt^\alpha)'}$. Now
$(\Tt^\alpha)'$ is the unique non-padded tree
satisfying these conditions. But since $q^\beta_\rho=P$,
$M^\beta||(\rho^+)^P=P||(\rho^+)^P$, so
$(\Tt^\alpha)'\ins(\Tt^\beta)'$, and for any
$E\neq\emptyset$ used by $\Tt^\beta$ but not by $\Tt^\alpha$, we have
$(\rho^+)^P\leq\lh(E)$.
Now let us show that $(\Tt^\beta)'\neq(\Tt^\alpha)'$, and letting $\gamma$ be
least such that
$E=E^{(\Tt^\beta)'}_\gamma$
is not used in $(\Tt^\alpha)'$, we
have $\lh(E)\leq\OR^P$. Suppose otherwise. Then $P\pins M^\beta=P^\beta_0$, and
letting
$i<k_\beta$
be least such that $\mu_i^\beta=\rho$, we have
$\gamma^\beta_i\geq(\rho^+)^P$, but then $P\pins P^\beta_i$
(if $j\leq i$ then either $P^\beta_j=P^\beta_0$ or
$\rho_\om(P^\beta_j)<\mu^\beta_i$). But $P=q^\beta_\rho\pins Q_i^\beta$, and
$\theta^\beta_i=(\rho^+)^P$. Therefore $P$ is
not $(A,B^\beta)$-valid. So
$k_\beta=i$,
contradiction.

Now let $\gamma$ be least such that $\rho<\lh(E^{\Tt^\beta}_\gamma)$. Then
$\rho\in\rhovec^\gamma$ and $P^\gamma_\rho=P$. This is by similar reasoning to
that in the previous paragraph.

So we may assume $\gamma=\alpha<\beta=\chi+1$. Now suppose
$\Uu\rest\chi+2=\Uu\rest(\gamma+1)\conc\Pp$, so $M^\Uu_{\chi+1}=M^\Uu_\gamma$.
By (\ref{a:Text(i)}) and (\ref{a:<limUpad}),(\ref{a:limUpad})($\iota=\chi+1$),
$\sigmavec^\gamma=\sigmavec^{\chi+1}\inter\lh(E)$, and
$\xi^{\chi+1}_\rho=\xi^\gamma_\rho$ for all $\rho\in\sigmavec^\gamma$. But
$\rho<\lh(E)$ and if $\rho\in\sigmavec^\gamma$ then $q^\gamma_\rho\neq
P^\gamma_\rho$, but $q^{\chi+1}_\rho=P^\gamma_\rho$, contradiction. So
$E^\Uu_\eps\neq\emptyset$ for some $\eps\in(\gamma,\chi]$. Let $\eps$ be least
such that $\gamma<\eps$, $E^\Uu_\eps\neq\emptyset$, $\eps+1\leq_\Uu\chi+1$, and
let $\delta=\Uu\pred(\eps+1)$. So either $\delta<\gamma$ or $\delta=\gamma'$,
where $\gamma'>\gamma$ is least such that $E^\Uu_{\gamma'}\neq\emptyset$. Now
$\kappa=\crit(i^\Uu_{\delta,\chi+1})\leq\rho$, by (\ref{a:Ubelow_crit}) and
(\ref{a:distinctions})($\iota=\chi+1$). Also
$\rho<\lh(E)\leq\mu^\eps\leq\nu(F^\Uu_\eps)$. So if $\eps<\chi$ then
$\rho<\crit(
i^\Uu_{\eps+1,\chi+1})$ by (\ref{a:E^UU(iv)}), so
$\rho\in\sigmavec^{\eps+1}$ and $q^{\eps+1}_\rho=q^{\chi+1}_\rho$, contradicting
(\ref{a:distinctions})($\iota=\chi+1$). So $\eps=\chi$. So by (\ref{a:E^UU}),
since $\kappa\leq\rho$ and $\rho\in\sigmavec^{\chi+1}$,
we have $\rho\geq\mu'$, where $\mu'$ is defined as there. But
$\rho<\lh(E)\leq\mu^\chi$, so $\rho=\mu'<\mu^\chi$.
Now let $\rho'=\max(\{0\}\un\sigmavec^\delta\inter\kappa)$. Then by
(\ref{a:E^UU}), $Q^{\chi+1}_{\rho'}=i^\Uu_{\delta,\chi+1}(Q^\delta_{\rho'})$ and
$q^{\chi+1}_\rho\pins Q^{\chi+1}_{\rho'}$, and so $q^{\chi+1}_\rho\pins
N'=i^{M^\Uu_\chi}_{E^\Uu_\chi}(Q^\chi_*)$.
But $P^\gamma_\rho=q^{\chi+1}_\rho$, and $E$ is on $\es_+(P^\gamma_\rho)$, and
so
on $\es^{N'}$. Therefore $E$ should have been retracted and not used in
$\Tt^{\chi+1}$, contradiction.

It is particularly in order to deal with the preceding situation that we use
retraction of extenders.

\begin{case} $\iota$ is a limit $\eta$.\end{case}

We must prove (\ref{a:general(iv)})($\iota=\alpha+\om$) and
(\ref{a:<limneat}),(\ref{a:<limUpad}),(\ref{a:<limUnonpad})($\iota=\eta$).

(\ref{a:<limneat}): We omit the proof.

(\ref{a:general(iv)}): Let $\delta=\delta(\Ttvec\rest\alpha+\om)$. Then
$\delta$ is a limit of limit cardinals of $M(\Ttvec\rest\alpha+\om)$, since
either $\Tt^{<\alpha+\om}$, or $\Uu\rest\alpha+\om$, is cofinally non-padded,
and
in the case that $\Tt^{<\alpha+\om}$ is eventually only padding, if
$\gamma<_\Uu\alpha+\om$ and $\kappa=\crit(i^\Uu_{\gamma,\alpha+\om})$, then
$\kappa\leq\lambda^\gamma$ is a limit cardinal of $M^\Uu_\gamma$ and
$Q^\gamma$,
and so by (\ref{a:general(ii)})($\iota<\eta$), $Q^\gamma|\kappa\ins
M(\Ttvec\rest\alpha+\om)$ and $\kappa$ is a limit cardinal there. Now,
$\lambda^\alpha\leq\lambda^{\alpha+n}\leq\mu^{\alpha+n}$ for all $n<\om$, by
(\ref{a:general(ii)}). Now suppose $\Uu\rest\alpha+\om$ is cofinally non-padded.
Then for every $\chi\in[\alpha,\alpha+\om)$ such that $E^\Uu_\chi\neq\emptyset$,
we have $\lambda^\alpha\leq\lh(F^\Uu_\chi)$. Since $M^\Uu_{\alpha+\om}$ is
wellfounded, there is such a $\chi<\alpha+\om$ with
$\lambda^\alpha<\crit(F^\Uu_\chi)\leq\lambda^\chi$.
It's similar if $\Tt^{<\alpha+\om}$ is cofinally non-padded.

(\ref{a:<limUpad}): This follows
(\ref{a:Text}),(\ref{a:<limUpad}),(\ref{a:limUpad})($\iota<\eta$)
and (\ref{a:<limneat})($\iota=\eta$).
Prove
(\ref{a:<limUpad(i)}) first; the others follow. (Note any descent has finite
length.)

(\ref{a:<limUnonpad}): This follows
(\ref{a:Ubelow_crit}),(\ref{a:general(ii)})($\iota<\eta$).

\begin{case} $\iota=\eta+1$ for a limit $\eta$.\end{case}

We must prove (\ref{a:general}),(\ref{a:Text}),(\ref{a:Ubelow_crit}) with
$\beta=\eta$, (\ref{a:distinctions}) with $\max(\alpha,\beta)=\eta$, and
(\ref{a:limUpad}),(\ref{a:limUnonpad}).

(\ref{a:limUpad}):
Let $\alpha$ be as in (\ref{a:<limUpad(ii)}). Then $\delta$ is a limit of cardinals of $M^\eta$, and of $Q^\alpha_\rho=Q^\alpha$,
since $\lh(E)$ is a cardinal of $Q^\alpha$ for each extender $E$
used by $\Tt^{<\eta}$, by (\ref{a:Text})($\iota=\eta$) and
(\ref{a:<limUpad}). Moreover,
$M^\eta|\delta=Q^\alpha|\delta$
by (\ref{a:<limUpad}). This gives the result.

(\ref{a:limUnonpad}):
We assume that $\Tt^{<\eta}$ is eventually only padding as the contrary case is
easier. 
However, there still may be cofinally many $\alpha<\eta$ such that
$E^{\Tt^{\alpha+1}}_\alpha\neq\emptyset$. We prove most of
(\ref{a:limUnonpad(v)}) and omit
the rest. Let $\gamma_0$ be least such that
$\Tt^{<\eta}=\Tt^{<\eta}\rest(\gamma_0+1)\conc\Pp$.
Let $A_0$ be the set of all $\beta\in[\gamma_0,\eta)$ such that
$\Tt^\beta=\Tt^{\gamma_0}\conc\Pp$. Then $A_0$ is cofinal in $\eta$. Let
$N=M^{\gamma_0}$. For each $\beta\in A_0$,
$M^\beta=N$. For $\beta_1<\beta_2$ with $\beta_1,\beta_2\in A_0$ we
have $d^{\beta_1}\rest\lambda^{\beta_1}=d^{\beta_2}\rest\lambda^{\beta_2}$. This
follows by induction on $\beta_2$, using
(\ref{a:E^UT(i)}),(\ref{a:general(ii)}),(\ref{a:limUnonpad(v)})($\iota<\eta$)
(note (\ref{a:limUnonpad(v)}) applies at every limit $\eta'\in(\gamma_0,\eta)$
as $\Tt^{<\eta'}=\Tt^{\gamma_0}\conc\Pp$).
So there is $\gamma\in A_0$ such that
$\rhovec^{\beta}\inter\lambda^{\beta}\sub\lambda^{\gamma}$ for all $\beta\in
A_0\inter[\gamma,\eta)$. But $\delta=\sup_{\beta<\eta}\lambda^\beta$, by
(\ref{a:general(ii)}),(\ref{a:general(iv)}).
It follows that $\gamma$ is as required. Now
use an argument like that for
(\ref{a:Ubelow_crit}); we omit the details.

(\ref{a:distinctions}): Suppose $P^\alpha_\rho=q^\beta_\rho$. By the argument
for (\ref{a:distinctions}) in the ``$\iota=\chi+2$'' case, we may assume
$\alpha<\beta=\eta$, and that argument shows that $\Tt^\eta$ uses some extender
$E$ such that $\rho<\lh(E)$. Therefore $\rho<\delta(\Ttvec\rest\eta)$. But then
by (\ref{a:limUpad(i)}),(\ref{a:limUnonpad(i)}), $q^\eta_\rho=q^\xi_\rho$ for
some $\xi<\eta$. So $P^\alpha_\rho=q^\xi_\rho$, contradicting
(\ref{a:distinctions})($\iota=\max(\alpha+1,\xi+1)$).

(\ref{a:general}),(\ref{a:Text}),(\ref{a:Ubelow_crit}): We omit the proof.

This completes the proof of Claim \ref{clm:agmt}.
\end{proof}

We can now show that the construction works.

\begin{clm}\label{clm:termination} The comparison terminates at some stage
$\zeta<\theta^+$.\end{clm}
\begin{proof}
Suppose not. Then we reach $\Tt=\Tt^{\theta^+}$ and $\Uu=\Uu\rest\theta^++1$.
Since $\mouseM,\cpmR$ have cardinality $<\theta^+$, both $\Tt\rest\theta^+$
and $\Uu\rest\theta^+$ are cofinally non-padded.
Let $\eta$ be some large ordinal and $\pi:H\to V_\eta$ elementary with $H$
transitive, $H$ of cardinality $\theta$,
$\crit(\pi)>\theta$, and $\Ttvec\rest(\theta^++1),\Uu$, etc, in $\rg(\pi)$.
Let $\kappa=\crit(\pi)$. As usual, $i^\Tt_{\kappa,\theta^+}$ and
$i^\Uu_{\kappa,\theta^+}$ both exist, have critical point $\kappa$, send
$\kappa$ to $\theta^+$, and agree over $M^\Tt_\kappa\inter M^\Uu_\kappa$. Let
$\alpha+1\in(\kappa,\theta^+]_\Tt$ be such that
$\crit(E^\Tt_\alpha)=\kappa$
and let $\beta+1\in(\kappa,\theta^+]_\Uu$ be such that
$\crit(E^\Uu_\beta)=\kappa$.
Since $\Tt$ is normal and since $\Uu$ does not move fine structural generators,
by Claim \ref{clm:agmt}(\ref{a:E^UU(iv)}), the extenders
$E^\Tt_\alpha$ and $E^\Uu_\beta$ are compatible over $\pow(\kappa)\inter
M^\Tt_\kappa\inter M^\Uu_\kappa$, through
$\nu=\min(\nu(E^\Tt_\alpha),\nu(F^\Uu_\beta))$, and
$\crit(i^\Uu_{\beta+1,\theta^+})\geq\nu(F^\Uu_\beta)$.

For all $\gamma\in[\kappa,\theta^+]$ we have $\pow(\kappa)\inter
M^\Tt_\kappa=\pow(\kappa)\inter M^\Tt_\gamma$ and $\pow(\kappa)\inter
M^\Uu_\kappa=\pow(\kappa)\inter M^\Uu_\gamma$. Also, letting $\kappa'\geq\kappa$
be least such that $E^\Tt_{\kappa'}\neq\emptyset$, we have
$M^\Tt_{\kappa'}=M^\Tt_\kappa$ and
$\lh(E^\Tt_{\kappa'})\geq(\kappa^+)^{M^\Tt_\kappa}$. So $\pow(\kappa)\inter
M^\Tt_\kappa\sub
M^\Tt_{\kappa'}||\lh(E^\Tt_{\kappa'})=Q^{\kappa'}||\theta^{\kappa'
}\in M^\Uu_{\kappa'}$, so $\pow(\kappa)\inter M^\Tt_\kappa\sub
M^\Uu_\kappa$.

Let $Q=M(\Tt\rest\theta^+)$. So $Q,M^\Tt_\kappa,M^\Tt_\alpha$ all
compute the same value for $\kappa^+$ and agree strictly below that point. Also
$E^\Tt_\alpha\notin Q$. By Claim
\ref{clm:agmt}(\ref{a:Ubelow_crit}) we have
that $Q^{\beta+1}_\rho=i^\Uu_{\kappa,\beta+1}(Q^\kappa_\rho)$, where
$\rho=\max(\{0\}\un(\sigmavec^\kappa\inter\kappa))$. Also,
$Q=i^\Uu_{\kappa,\theta^+}(Q^\kappa_\rho|\kappa)$, which, again by Claim
\ref{clm:agmt}(\ref{a:Ubelow_crit}), implies that
$\rho=\max(\{0\}\un(\sigmavec^{\beta+1}\inter\crit(i^\Uu_{\beta+1,\theta^+}))$.
Since
$\crit(i^\Uu_{\beta+1,\theta^+})\geq\nu(F^\Uu_\beta)$, we have
$Q|\nu(F^\Uu_\beta)=Q^{\beta+1}_\rho|\nu(F^\Uu_\beta)$ and
$Q||\nu(F^\Uu_\beta)=Q^\beta|\nu(F^\Uu_\beta)$. In particular,
$Q||(\kappa^+)^{Q^\beta}=Q^\beta|(\kappa^+)^{Q^\beta}$. However, we might have
$(\kappa^+)^{Q^\beta}<(\kappa^+)^Q$.

Since $F^\Uu_\beta\rest\nu(F^\Uu_\beta)\sub E^\Uu_\beta$, the compatibility of
$E^\Tt_\alpha$ with $E^\Uu_\beta$ implies that
if $\nu(E^\Tt_\alpha)\leq\nu(F^\Uu_\beta)$ then
$E^\Tt_\alpha\rest\nu(E^\Tt_\alpha)\sub F^\Uu_\beta$, and
if $\nu(F^\Uu_\beta)\leq\nu(E^\Tt_\alpha)$ then
$F^\Uu_\beta\rest\nu(F^\Uu_\beta)\sub E^\Tt_\alpha$.
But maybe $\nu(F^\Uu_\beta)<(\kappa^+)^Q$, in which case
$F^\Uu_\beta$ is not total over $Q$.

\begin{sclm} $\nu(E^\Tt_\alpha)\geq\nu(F^\Uu_\beta)$ and
$\alpha>\beta$.\end{sclm}
\begin{proof}
Suppose $\nu(E^\Tt_\alpha)<\nu(F^\Uu_\beta)$. Then
$(\kappa^+)^{M^\Tt_\alpha}\leq\nu(E^\Tt_\alpha)<\nu(F^\Uu_\beta)$ and
$E^\Tt_\alpha\rest\nu(E^\Tt_\alpha)$ is a proper, non-type Z initial segment of
$F^\Uu_\beta$. So $E^\Tt_\alpha\in i^\Uu_{\kappa,\beta+1}(Q^\kappa_\rho)$, by
Claim \ref{clm:agmt}(\ref{a:E^UU(vii)}). But
$\crit(i^\Uu_{\beta+1,\theta^+})\geq\nu(F^\Uu_\beta)$, so
$\crit(i^\Uu_{\beta+1,\theta^+})>\lh(E^\Tt_\alpha)$, so $E^\Tt_\alpha\in Q$,
contradiction.

Now $\alpha\neq\beta$ by construction. If $\alpha<\beta$ then by Claim
\ref{clm:agmt}(\ref{a:Text(ii)}),
$\nu(E^\Tt_\alpha)<\lh(E^\Tt_\alpha)\leq\nu(F^\Uu_\beta)$, a contradiction,
which proves the subclaim.
\end{proof}

So $\nu=\nu(F^\Uu_\beta)$. Let
$N=M^\Tt_\alpha|\lh(E^\Tt_\alpha)=M^\alpha|\lh(E^\Tt_\alpha)$.

\begin{sclm}\label{sclm:isc} Either (a) $F^\Uu_\beta\in\es_+^N$ or else (b)
$F^\Uu_\beta$ is either type 1 or type 3, $N|\nu$ is active and
$F^\Uu_\beta\in\es^{\Ult(N||\nu,\es^N_\nu)}$.
\end{sclm}
\begin{proof}
We know that $F^\Uu_\beta\rest\nu\sub E^\Tt_\alpha$. So if $(\kappa^+)^Q\leq\nu$
then the desired conclusion follows the initial segment condition.

Suppose $\nu<(\kappa^+)^Q$. Then $F^\Uu_\beta$ is type 1. For otherwise,
$\gamma=(\kappa^+)^{Q^\beta}<\nu$, so $\gamma$ is a cardinal of
$Q^{\beta+1}_\rho$, contradicting that
$(\kappa^+)^{Q^{\beta+1}_\rho}=(\kappa^+)^Q$. So $F^\Uu_\beta$ is a partial
normal measure derived from $E^\Tt_\alpha$, inducing the type 1 premouse
$R=Q^\beta$ such that $\nu=(\kappa^+)^R<(\kappa^+)^N$ and $R|\nu=N||\nu$. We now
use \cite[4.11, 4.12, 4.15]{thesis} to yield the conclusion of the
subclaim.\footnote{It seems one might try to deduce
\cite[4.11, 4.12, 4.15]{thesis} from the $n=0$ condensation given in
\cite[pp.87,88]{fsit}. That is, let $E=\es^N_\gamma$ be the type 1 initial
segment of $E^\Tt_\alpha$. Using a restriction of the factor map
$j:\Ult_0(Q|\nu,F^Q)\to\Ult_0(N|(\kappa^+)^N,E)$, we get a $\Sigma_0$-elementary
$\pi:Q\to N|\gamma$, with $\crit(\pi)=\nu$ and $\pi(\nu)=(\kappa^+)^N$.
Moreover, $\rho_1^Q\leq\nu$. However, $\pi$ need not be $\Sigma_1$-elementary,
even for formulas without parameters, so $\pi$ might not even 
be a weak $0$-embedding (for instance, if $F=F^Q$ is the least partial measure
derived from $E$ such that $F$ is on $\es^N$). So the $n=0$ condensation of
\cite[pp.87,88]{fsit} does not apply.} Since $\mouseM$ is typical, these
apply to $N$. However,
if $N|\nu$ is active with a type 3 extender, then we must verify that
\cite[4.15]{thesis} applies; that is, we must verify that
$R||\OR^R=\Ult(N||\nu,\es^N_\nu)||\OR^R$.
Well, $\Tt^\beta$ and $\Tt^\alpha$ use the same extenders $E$ such that
$\lh(E)<\nu$. However, $N|\nu$ is active while $M^\beta|\nu$ is not,
so $\Tt^\beta$ uses $\es^N_\nu$. Moreover, $\nu$ is the largest cardinal of $R$,
and $R||\OR^R=M^\beta||\OR^R$. Therefore $\Tt^\beta$ uses no
extenders $E$ such that $\nu<\lh(E)<\OR^R$ and
$M^\beta||\OR^R=\Ult(N||\nu,\es^N_\nu)||\OR^R$. So
\cite[4.15]{thesis} applies.

This completes the proof of the subclaim.\end{proof}

\begin{sclm}
$F^\Uu_\beta\notin\es_+(M^\beta)$.
\end{sclm}
\begin{proof}
 Suppose $F^\Uu_\beta\in\es_+(M^\beta)$. Then 
$P^\beta|\theta^\beta=M^\beta|\lh(F^\Uu_\beta)=Q^\beta|\theta^\beta$ is active,
but
$(\dagger)^\beta$
fails, so by \ref{rem:background_exists},
$M^\beta|\lh(F^\Uu_\beta)$ is not $(A,B^\beta)$-valid. But $A$ is
bounded in $\kappa$. So $F^\Uu_\beta$ induces an extender algebra axiom which is
not satisfied by $(A,B^\beta)$, which gives a contradiction as usual,
proving the subclaim.
\end{proof}

\begin{sclm}\label{sclm:F^Uu_beta_used}$\Tt^\beta$ uses $F^\Uu_\beta$.\end{sclm}
\begin{proof}
Suppose Subclaim \ref{sclm:isc}(a) holds. Then $\Tt^\alpha$ and $\Tt^\beta$ use
the same extenders $E$ such that $\lh(E)<\lh(F^\Uu_\beta)$.
Since $F^\Uu_\beta\in\es_+^N$ but
$F^\Uu_\beta\notin\es_+(M^\beta)$,
$\Tt^\beta$ uses $F^\Uu_\beta$.

If Subclaim \ref{sclm:isc}(b) holds it is similar but there are
$\delta_0<\delta_1<\beta$ such that $E^{\Tt^\beta}_{\delta_0}=\es^N_\nu$,
$E^{\Tt^\beta}_{\delta_1}=F^\Uu_\beta$, and $E^{\Tt^\beta}_\delta=\emptyset$ for
all $\delta\in(\delta_0,\delta_1)$.

This completes the proof of the subclaim.
\end{proof}

But Subclaim \ref{sclm:F^Uu_beta_used} contradicts Claim
\ref{clm:agmt}(\ref{a:Text(ii)}) at stage $\beta+1$, proving
Claim \ref{clm:termination}.
\end{proof}

By Claim \ref{clm:termination} we have $\zeta<\theta^+$ such that
$(\dagger)^\zeta$ holds. Let $\Tt=\Tt^\zeta$ and $\Uu=\Uu\rest\zeta+1$.

\begin{clm}\label{clm:drop} Either $b^\Tt$ does not drop in model or
$i^\Uu(N_\conlength)\ins
\I^\Tt$.\end{clm}

\begin{proof}
We first relate cores of models on $\Tt$ to the structures arising in the
comparison.
\begin{sclm*}\label{sclm:cores} Let $\alpha+1<\lh(\Tt^\beta)$ and let
$\eps=\Tt^\beta\pred(\alpha+1)$. Let $\kappa=\crit(E^{\Tt^\beta}_\alpha)$. If
$\kappa<\min(\rhovec^\eps)$ then $\Tt^\beta$ does not drop in model at
$\alpha+1$ \textup{(}here $\min(\emptyset)=\infty$\textup{)}. If
$\min(\rhovec^\eps)\leq\kappa$ then $M_{\alpha+1}^{*\Tt^\beta}=P^\eps_\rho$
where $\rho=\max(\rhovec^\eps\inter(\kappa+1))$.\end{sclm*}
\begin{proof} This follows \ref{rem:E^Tt_eps_neq_emptyset}.
\end{proof}

Now suppose the claim fails. So $b^\Tt$ drops in model, and by
\ref{rem:dagger_holds}, we may assume that
$\I^\Tt=M^\zeta=Q^\zeta$ and $\xi^\zeta<\xi^\zeta_0$.
Let $\eps<\lh(\Tt)$ be such that $\core_\om(\I^\Tt)^{\unsq}\pins M^\Tt_\eps$. Let
$\rho=\rho_\om(\I^\Tt)$.
By the Subclaim, $\rho\in\rhovec^\eps$ and
$\core_\om(\I^\Tt)=\core_0(P^\eps_\rho)$. We have
$\I^\Tt=Q^\zeta=Q^\zeta_{\sigma}$ for some $\sigma\in\sigmavec^\zeta$ (since
$\xi^\zeta<\xi^\zeta_0$). So
$\core_0(q^\zeta_{\sigma})=\core_\om(Q^\zeta_{\sigma})=\core_0(P^\eps_\rho)$.
Therefore $q^\zeta_{\sigma}=P^\eps_\rho$, contradicting Claim
\ref{clm:agmt}(\ref{a:distinctions}), and completing the proof of Claim
\ref{clm:drop}.
\end{proof}

We have shown that $\Tt,\Uu$ satisfy conditions
\ref{i:proof_main1_a}-\ref{i:proof_main1_c}. We now refine this to complete the
proof of \ref{thm:main1}:

\begin{clm}\label{clm:core} There is $\eps\leq\zeta$ such that
$(\Tt\rest\eps+1,\Uu\rest\eps+1)$ satisfies the requirements of
\ref{thm:main1}.\end{clm}

\begin{proof}
If $i^\Uu(N_\conlength)\pins \I^\Tt$ then
$N^\zeta_\conlength$ is
$\om$-sound, and we just
use $\eps=\zeta$. So assume that $\I^\Tt=N_\alpha^\zeta$ for some
$\alpha\leq\conlength^\zeta$. Let $b=b^\Tt$. If $b$ does not drop in model or
degree, again we use $\eps=\zeta$.
So assume that $b$ drops in model or degree. We have two cases to deal with:
(i) either $b$ drops in model or [$\alpha=\conlength^\zeta$ and $m>n$];
(ii) otherwise.

We assume we are in case (ii), but the proof in case (i) is
almost the same. So $b$ drops in degree but not in model, and
$(\alpha,m)\leq_\lex(\conlength^\zeta,n)$.
Now
$\core_m(N_\alpha^\zeta)=\core_m(\I^\Tt)=M^\Tt_\gamma$ for some
$\gamma<\lh(\Tt)$. Let $\gamma$ be least such and $\gamma'$
greatest such
(so $\gamma'\geq\gamma$ is least such that
$E=E^\Tt_{\gamma'}\neq\emptyset$). Let $\rho=\rho_m^{\I^\Tt}$ and let
$\tau=(\rho^+)^{M^\Tt_\gamma}$. Let
$\beta\in c=b^\Uu$ be least such that either
$i^\Uu_{\beta,\zeta}=\id$ or $\crit(i^\Uu_{\beta,\zeta})>\rho$. Let
$\beta'$ be largest such that $M^\Uu_{\beta'}=M^\Uu_\beta$. Let
$\eps=\max(\gamma,\beta)$.
We will show that this works.

Since $b$ does not drop in model and $\crit(i^\Tt_{\gamma,\zeta})\geq\rho$,
we have $\lh(E)\geq\tau$, and if
$\lh(E)=\tau$ then
$E$ is type 2. Since type 2 extenders are not relevant to
validity, therefore $M^\Tt_\gamma|\tau$ is
$(A,B^{\gamma'})$-valid.

By choice of $\beta$ and elementarity,
$M^\Tt_\gamma=\core_m(N^\beta_{\alpha'})$ for some
$\alpha'\leq\conlength^\beta$.

\begin{sclm*} $\eps=\min(\beta',\gamma')$.
\end{sclm*}
\begin{proof}
Since $\beta\leq\beta'$ and $\gamma\leq\gamma'$, we just have to rule out
the possibility that either $\beta\leq\beta'<\gamma$ or
$\gamma\leq\gamma'<\beta$.

Suppose $\beta\leq\beta'<\gamma$. In particular, $\gamma\neq 0$ and
$\lh(\Uu)>\beta'+1$, so $E^\Uu_{\beta'}\neq\emptyset$. Now
$\rho<\crit(i^\Uu_{\beta,\zeta})\leq\lambda^{\beta'}$ by Claim
\ref{clm:agmt}(\ref{a:E^UU(iii)}),
but because $\beta'<\gamma$, Claim \ref{clm:agmt}(\ref{a:general(ii)}) gives
that $\lambda^{\beta'}<\rho_m(M^\Tt_\gamma)$, contradiction.

Now suppose $\gamma\leq\gamma'<\beta$. Let $\xi\geq\gamma'$ be least such
that $E^\Uu_\xi\neq\emptyset$ and $\xi+1\in c$. Then $\xi>\gamma'$ since
$E^\Tt_{\gamma'}\neq\emptyset$. Since
$\tau\leq\lh(E)$, therefore by Claim
\ref{clm:agmt},
$\tau\leq\nu(F^\Uu_\xi)\leq\crit(i^\Uu_{\xi+1,\zeta})$. So
$\beta=\xi+1$. Let $\sigma=\Uu\pred(\xi+1)$, let $j=i^\Uu_{\sigma,\xi+1}$ and
let $\kappa=\crit(j)$. Let $\eta$ be such that
$N_\eta^\sigma=M^\Tt_\gamma|\kappa$. Let $\eta'>\eta$ be
least such that either $\eta'=\infty$ or for some $k$,
$(\eta',k)<_\lex(\conlength^\sigma,n)$ and
$\rho_{k+1}(N_{\eta'}^{\sigma})<\kappa$.
So $j(\eta')$ is defined similarly in
$M^\Uu_{\xi+1}$. Now $M^\Tt_\gamma||\tau=
N_{\alpha'}^{\xi+1}||\tau$ and
$\rho_m(N_{\alpha'}^{\xi+1})=\rho$.
By Claim \ref{clm:agmt}, and since $E$ was never retracted
after stage $\gamma'$,
$j(M^\Tt_\gamma|\kappa)|\lh(E)=M^\Tt_\gamma||\lh(E)$, and $\rho$ is a cardinal
in $j(M^\Tt_\gamma|\kappa)$.

Suppose
$\tau=(\rho^+)^{j(M^\Tt_\gamma|\kappa)}$.
Then $N_{\alpha'}^{\xi+1}$ witnesses that
$j(\eta')\neq\infty$, so $\eta'\neq\infty$. But
$j(\eta')<\alpha'$, because $\rho\notin\rg(j)$. Moreover,
$\rho_\om(N^\sigma_{\eta'})<\kappa$. But
$N_{\alpha'}^{\xi+1}|\kappa=j(N_{\eta'}^{\sigma})|\kappa$,
which leads
to contradiction.

So $\tau<(\rho^+)^{j(M^\Tt_\gamma|\kappa)}$.
But then the properties of $N_{\alpha'}^{\xi+1}$, and that $\rho$ is a cardinal
of $j(M^\Tt_\gamma|\kappa)$, give that $M^\Tt_\gamma\pins
j(M^\Tt_\gamma|\kappa)$, contradicting the fact that
$M^\Tt_\gamma||\lh(E)\ins j(M^\Tt_\gamma|\kappa)$.
This proves the subclaim.
\end{proof}

Now by the
subclaim, $B^\eps=B^{\gamma'}\inter(\rho+1)^3$,
so $M^\Tt_\gamma|\tau$
is $(A,B^\eps)$-valid, and the claim, and properties
\ref{thm:main1}(a),(b), follow.
\end{proof}

This completes the proof of the theorem.\qed
\end{proof}

We finish with one corollary to the foregoing proof, which answers
a question of Nam Trang and Martin Zeman. For simplicity we assume that
$m=n=0$.

\begin{cor}
 Let $\mouseM$, etc, be as in the statement of \ref{thm:main1}, and assume
$m=n=0$. Let
$\Tt,\Uu$ be constructed as in its proof. Let $\zeta+1=\lh(\Tt)=\lh(\Uu)$. Let
$\hat{\cpmR}=M^\Uu_\zeta$ and $\hat{\construction}=i^\Uu_\zeta(\construction)$.
Let $\hat{\Tt},\hat{\Uu}$ be given
by applying the same construction to
$(\mouseM,(\hat{\cpmR},\hat{\construction}))$. Then $\hat{\Tt}$ is the
non-padded tree $\Tt'$ equivalent to $\Tt$ and $\hat{\Uu}$ is only padding.
\end{cor}
\begin{proof}
We adopt the notation of the proof of \ref{thm:main1} regarding the construction
of $\Tt,\Uu$. Let $\hat{M}^\alpha$, $\hat{R}^\alpha$, etc, be the corresponding 
notation regarding the construction of $\hat{\Tt},\hat{\Uu}$. Note that since
$m=n=0$, Claim \ref{clm:core} of the proof of \ref{thm:main1} is trivial and its
proof does nothing.

\begin{clm*}
For each $\alpha<\lh(\Tt')$, $\hat{\Tt}^\alpha=\Tt'\rest\alpha+1$ and
$\hat{\Uu}\rest\alpha+1$ is pure padding.\end{clm*}

\begin{proof}The proof is by induction on $\alpha$. Suppose it holds for
$\alpha$, and $\lh(\Tt')>\alpha+1$. Let $B$ be the neat code for
$(\Tt,\Uu)$. Let $P=\I^\Tt$ or $P=i^\Uu(N_\conlength^\construction)$, whichever
is smaller. Because $P$ is $(A,B)$-valid, and because of the inductive
hypothesis, and the fact that $\hat{\Tt}^\alpha$ and $\Tt$ are coded (via their
neat codes) in a manner that ignores padding, the proof that
$\hat{\Tt}^{\alpha+1}=\Tt'\rest\alpha+2$ will not break down due to
$(A,\hat{B}^\alpha)$-invalidity. (Since $\hat{\Uu}\rest\alpha+1$ is pure
padding, this portion of $\hat{B}^\alpha$ is also not a problem.)

Now let $\gamma=\lh(E^{\Tt'}_\alpha)$. Let $\beta$ be such that
$E^\Tt_\beta=E^{\Tt'}_\alpha$. Then $\hat{d}^\alpha=d^\beta$
and $\hat{e}^\alpha=e^\zeta\rest\gamma$ and $\hat{P}^\alpha=P^\beta$.
(Recall $\zeta+1=\lh(\Tt)$; see \ref{dfn:descent} for the definition of
$d^\beta$, $e^\zeta$, etc.) This follows by an argument like in the proof of
Claim \ref{clm:agmt}(\ref{a:Ubelow_crit}), combined with the above observations
regarding validity, and using that $M^\Tt_\beta||\lh(E^\Tt_\beta)$ is a cardinal
segment of $\I^\Tt$ and $Q^\zeta$ (by Claim
\ref{clm:agmt}(\ref{a:Text})). Also, $\hat{Q}^\alpha|\hat{\theta}^\alpha$ is
passive since $\hat{Q}^\alpha|\hat{\theta}^\alpha\ins Q^\zeta$. So
$E^{\hat{\Tt}}_\alpha=E^{\Tt'}_\alpha$, as required.

The claim easily follows.
\end{proof}

So we reach stage $\hat{\zeta}$, at which we have
$\hat{M}^{\hat{\zeta}}=M^\zeta$ and $\hat{\Uu}\rest\hat{\zeta}+1$ is pure
padding. But then $(\hat{\dagger})^{\hat{\zeta}}$ holds since
$(\dagger)^\zeta$ does. This completes the proof.\qed \end{proof}

\bibliography{biblio}

\begin{thebibliography}{1}

\bibitem{it}
Donald Martin and John~R. Steel.
\newblock Iteration trees.
\newblock {\em Journal of the American Mathematical Society}, 7(1):1--73,
  January 1989.

\bibitem{fsit}
William Mitchell and John~R. Steel.
\newblock {\em Fine structure and iteration trees}.
\newblock Number~3 in Lectures Notes in Logic. Springer-Verlag, 1994.

\bibitem{thesis}
Farmer Schlutzenberg.
\newblock {\em Measures in mice}.
\newblock PhD thesis, University of California, Berkeley, 2007.

\bibitem{derived}
John~R. Steel.
\newblock Derived models associated to mice.
\newblock In C.T. Chong, editor, {\em Computational Prospects of Infinity -
  Part I: Tutorials}, Lecture Notes. World Scientific Publishing Company,
  Incorporated, 2008.

\bibitem{outline}
John~R. Steel.
\newblock An outline of inner model theory.
\newblock In Matthew Foreman and Akihiro Kanamori, editors, {\em Handbook of
  set theory}, volume~3, chapter~19. Springer, first edition, 2010.

\end{thebibliography}
\bibliographystyle{plain}

\end{document}